\documentclass[final]{siamltex}
\newcommand{\blind}{1}
\usepackage{epsfig}
\usepackage{amsmath,amssymb}
\usepackage{xcolor}
\usepackage{graphicx}
\usepackage{epstopdf}
\usepackage{algorithmic,algorithm}
\usepackage{enumitem} 

\usepackage{mathdots,array,booktabs,mathrsfs}
\usepackage{multirow}
\usepackage{mathrsfs}
\usepackage{appendix}
\usepackage{amsfonts}
\usepackage{amsmath}
\usepackage{latexsym}
\usepackage{amssymb}
\usepackage{color}
\usepackage{amssymb,amsmath,array,epsf}
\usepackage{mathdots}
\usepackage{epsfig}
\usepackage{picinpar,graphicx}
\usepackage[caption=false]{subfig}  
\usepackage{url}

\newtheorem{example}{Example}
\newtheorem{rem}{Remark}[section]

\title{Precision-induced Adaptive Randomized Low-Rank Approximation for SVD and Matrix Inversion} 

\author{
Weiwei Xu \footnotemark[1]
,\ Weijie Shen\footnotemark[2]
,\and Zhengjian Bai\footnotemark[3], \and Chen Xu\footnotemark[4]}
\begin{document}
\maketitle

\renewcommand{\thefootnote}{\fnsymbol{footnote}}

\footnotetext[1]{School of Mathematics and Statistics, Nanjing University of Information Science and Technology, Nanjing 210044, P. R. China and School of Mathematics and Statistics, Xi'an Jiaotong University, Xi'an 710049, P. R. China (wwxu@nuist.edu.cn). The research of this author was partially supported by the National Key R\&D Program of China (2025YFA1016400).}
\footnotetext[2]{School of Mathematics and Statistics, Nanjing University of Information Science and Technology, Nanjing 210044, P. R. China (swj@nuist.edu.cn). The research of this author was partially supported by the Postgraduate Research \& Practice Innovation Program of Jiangsu Province under grants KYCX24\_1402.}
\footnotetext[3]{School of Mathematical Sciences,  Xiamen University, Xiamen 361005, P. R. China (zjbai@xmu.edu.cn). The research of this author was partially supported by the National Natural Science Foundation of China grant 12371382 and Fujian Provincial Natural Science Foundation of China grant 2025J01030.}
\footnotetext[4]{Department of Mathematics and Fundamental Research Pengcheng Laboratory, Shenzhen 518108, P. R. China and School of Mathematics and Statistics, Xi'an Jiaotong University, Xi'an 710049, P. R. China (cx3@xjtu.edu.cn). The research of this author was partially supported by National Key R\&D Program of China (2023YFA1008703) and Major Key Project of PCL (PCL2024A06).}

\renewcommand{\thefootnote}{\arabic{footnote}}

\slugger{mms}{xxxx}{xx}{x}{x--x}


\begin{abstract}
Singular value decomposition (SVD) and matrix inversion are ubiquitous in scientific computing. Both tasks are computationally demanding for large scale matrices. Existing algorithms can approximatively solve these problems with a given rank, which however is unknown in practice and requires considerable cost for tuning. In this paper, we tackle the SVD and matrix inversion problems from a new angle, where the optimal rank for the approximate solution is explicitly guided by the distribution of the singular values. Under the framework, we propose a precision-induced random re-normalization procedure for the considered problems without the need of guessing a good rank. The new algorithms built upon the procedure simultaneously calculate the optimal rank for the task at a desired precision level and lead to the corresponding approximate solution with a substantially reduced computational cost. The promising performance of the new algorithms is supported by both theory and numerical examples.
\end{abstract}

\begin{keywords}
Random re-normalization, nearly low-rank, Gaussian random matrix
\end{keywords}

\begin{AMS}
65F15, 65F99
\end{AMS}

\pagestyle{myheadings}
\thispagestyle{plain}
\markboth{Weiwei Xu, \ Weijie Shen,\ Zhengjian Bai, and Chen Xu}{An approach for accelerating SVD and matrix inversion}

\section{Introduction}
\setcounter{equation}{0}

Singular value decomposition (SVD) and matrix inversion are fundamental tools in scientific computing, with widespread applications \cite{o3} in areas such as the Kronecker canonical form of a general matrix \cite{o6}, the linearly constrained least-squares problem \cite{o11,o15}, the general Gauss-Markov linear model \cite{o13}, real-time signal processing \cite{o9,o14}, precoding channel information matrix in wireless communication \cite{oo1,oo8} and so on. 
To compute the  SVD of an $m\times n$ matrix $A$, the traditional algorithms such as Golub–Reinsch SVD \cite{gr1970} typically require a cost of \(O(\min\{mn^2,m^2n\})\) flops.
The iterative algorithms such as Lanczos \cite{jia2003,jia2010} and Jacobi-Davidson \cite{h2001} largely improve computational efficiency by using matrix-vector products; however, they often require careful tuning and may converge slowly to internal singular values.
To compute the matrix inversion of an $n\times n$ matrix $A$, it is known that the classical Gaussian elimination- or adjugate-matrix-based methods typically require a cost of $O(n^3)$ flops \cite{golub12,g1955}.
The block-inversion techniques elevate the classical methods by exploiting sparsity or structure via Schur complements, but they may also introduce a significant overhead cost when the block dimensions are small.
LU decompositions require forward or backward substitution (approximately $2n^3/3$ flops) and may be unstable for ill‑conditioned matrices, while QR‑based inversion improves stability at the cost of additional orthogonalization operations.


In real-world applications, many SVD and matrix inversion tasks need to be carried out on large-scale matrices with potentially ill-conditions \cite{oo1,oo8}. In such situations, the performance of the aforementioned methods may be less satisfactory due to the trade-off between computational cost and numerical precision; the most accurate method usually comes with the highest cost.
In practice, however, the true rank of the matrix is usually unknown a priori, and estimating it accurately often incurs additional computational overhead. In contrast, the desired precision or allowable error is typically known in advance from application requirements. Therefore, a more practical strategy is to drive the computation by the target precision and adaptively determine the necessary rank. For example, in image compression and denoising, the truncation level is chosen according to a specified PSNR or reconstruction error \cite{hcwy2024}; in wireless MIMO channel estimation and hybrid precoding, only singular modes that capture a given percentage of channel energy are preserved to satisfy performance constraints \cite{lch2022}; and in inverse or regularized problems, the truncation parameter is selected based on the target reconstruction accuracy or regularization tolerance \cite{abf2025}.
This calls for a more practical strategy: finding a "good-enough" method that meets a desired precision.

With the above thinking, we aim to design precision-induced methods for efficiently computing SVD and matrix inversion. It is observed that many real-world matrices come with a nearly low rank structure, where most spectral energies are concentrated in a small number of dominant singular values \cite{a2025,h1975,h1970,k2013,guo2024,xu2024}. This common characteristic readily indicates a natural definition of precision as below.
\begin{definition}[\(\epsilon\)-rank]\label{Ass1}
Let $A \in \mathbb{C}^{m \times n}$  with the singular values $\sigma_1(A)\ge \sigma_2(A)\ge\cdots\ge\sigma_{p}(A)\ge 0$, where $p=\min\{m,n\}$.
For any $0\le \epsilon <1\), we define the \(\epsilon\)-rank of \(A\) as
\[
r_{\epsilon}=\big\{\mbox{$\min(s) : 1 \leq s \leq p,\; {(\sum_{i=1}^{s}\sigma_i^2(A))}/({\sum_{i=1}^{p}\sigma_i^2(A)})\geq 1-\epsilon$}\big\},
\]
where \(r_\epsilon\) is a desired precision level to be specified by users and it is an approximated matrix rank ($\epsilon$-rank) induced by a given tolerance $\epsilon$. A moderate $\epsilon$ often leads to $r_{\epsilon}\ll p$, which enables efficient SVD on $A$ at the given precision level.
\end{definition}


The concept of $\epsilon$-rank provides a quantitative criterion for precision-controlled low-rank approximation. Building on this definition, randomized algorithms are employed to compute low-rank approximations that adapt to the effective numerical rank of matrices. In practice, these algorithms can be divided into two cases: when the target rank is known, a numerical rank is specified to obtain an $\epsilon$-rank approximation; when the rank is unknown, the approximation is adaptively constructed to satisfy a given tolerance $\epsilon$.
The basic idea of employing the $\epsilon$-rank for efficient SVD  can be summarized as follows.
First, a random sampling procedure is applied to reduce the dimensionality of the input matrix. In this step, an improved sampling strategy is adopted, where the Gram–Schmidt orthogonalization process is modified to adaptively determine the effective subspace dimension. Specifically, the diagonal elements of the upper-triangular matrix generated during orthogonalization are compared with the prescribed tolerance $\epsilon$; once the accumulated spectral energy exceeds $1-\epsilon$, the corresponding subspace is identified as the $\epsilon$-rank of the matrix, leading to a low-rank approximation that satisfies the given precision requirement.
Second, a complete singular value decomposition is performed on the reduced matrix.
Finally, the decomposition is projected back to the original space to obtain an approximate SVD of the original matrix.

Randomized algorithms offer a path to \(O(m n r_\epsilon)\) complexity without sacrificing accuracy \cite{92,11,gr1970,9,h2001,jia2003,jia2010,12,13,14,15,16,10,yu2018}.
Halko et al. \cite{9} reduced the cost by proposing randomized SVD (RSVD). Given a target subspace dimension $\ell$, RSVD draws a Gaussian random matrix \(\Omega\in\mathbb R^{n\times\ell}\), forms the sample matrix \(Y = A\Omega\), and computes an $m\times \ell$ matrix \(Q = \mathrm{orth}(Y)\) via QR decomposition, whose columns form an orthonormal basis for the range of $Y$. The approximate decomposition \(A\approx QQ^\mathrm{H}A\) is then used for downstream tasks.  While this fixed-$\ell$ strategy is simple and highly parallelizable, it suffers from two key drawbacks: (i) one must choose $\ell$ in advance, typically by oversampling the unknown rank by a heuristic amount (e.g.\ \(\ell = r + 10\)), which can lead to excessive computation if \(r_\epsilon\) is overestimated; (ii) if $\ell$ is underestimated, the approximation error may exceed the desired tolerance, necessitating a full rerun with a larger $\ell$–an expensive ``trial-and-error'' process.
To remove the need for choosing \(\ell\) in advance, the adaptive randomized range finder (ARRF) algorithm \cite{9} extends RSVD by adaptively sampling columns until a target error is met. This removes the requirement to predefine \(\ell\). At each iteration it augments \(Y\), orthogonalizes against the existing \(Q\) (via Schmidt orthogonalization), and checks whether the projected norm \(\|A - QQ^\mathrm{H}A\|\) is below the tolerance. On the downside, its column-by-column QR orthogonalization entails repeated matrix–vector products and QR decompositions. The orthogonal columns of $Q$ obtained are usually larger than $r_\epsilon$. As a result, the overall runtime may become longer.
More recently, Yu et al. \cite{yu2018} proposed the randQB‑EI algorithm, which samples and orthogonalizes in blocks and uses the Frobenius norm for efficient error estimation. This approach balances fixed-precision goals with reduced communication cost. However, updating the $B$-factor during each iteration still incurs \(O(mnr)\) flops.

Given the above, we propose a precision-induced random re-normalization procedure for efficient computation of SVD and matrix inverse. 
The proposed method is based on a Gaussian random matrix. After each generation and projection of a Gaussian random vector, the magnitude of the diagonal entry \( b_{kk} \) of the $B$-factor in the QB decomposition is directly used to determine whether the precision \( \epsilon \) has been achieved. This enables the $\epsilon$-rank and approximate basis of the matrix to be adaptively identified, and the SVD and matrix inverse can be constructed to arbitrary precision. 
This paper theoretically gives necessary and sufficient conditions to prove that the expectation of $b_{k+1,k+1}^2$ satisfies $\mathbb{E} b_{k+1,k+1}^2\leq\| A\|_{\mathrm{F}}^{2}\epsilon$, if and only if $\| A-\hat{A}\|_{\mathrm{F}}/\| A\|_{\mathrm{F}} \leq \sqrt{\epsilon}$, where $\mathbb{E}\xi$ denotes the expectation of a random variable $\xi$. In this case, we obtain $(\sum_{i=1}^{k}\sigma_{i}^{2}(A))/(\sum_{i=1}^{p}\sigma_{i}^{2}(A))$ $\geq 1-\epsilon$ and thus the $\epsilon$-rank of $A$ is $r_{\epsilon}=k$.
Compared to the ARRF algorithm in \cite{9} and the randQB-EI algorithm in \cite{yu2018}, the proposed approach avoids the computation of the full matrix \( B \), requires no residual estimation or manual parameter tuning, and relies solely on the lower bound of the retained energy ratio $1-\epsilon$ or the upper bound of the relative error accuracy \( \epsilon \). Theoretically, the relative reconstruction error,  singular value error, and matrix inverse error are all guaranteed to be \( O(\sqrt{\epsilon}) \).
We prove that the new methods can achieve a similar error bound with the traditional method with only $O(mnr_{\epsilon})$ complexity, lower than the traditional $\min\{O(m^2n),O(mn^2)\}$ complexity when $r_{\epsilon}$ is unknown. In typical near-low-rank scenarios, our proposed SVD algorithm achieves over a tenfold speedup compared to MATLAB’s built-in economy-sized SVD (eSVD). Similarly, our matrix inverse algorithm outperforms MATLAB’s built-in matrix inverse (BII) function in terms of computational efficiency. All those merits make the new method effective and efficient in the varying-precision approximation problem. The contributions of our paper are listed as follow:


\begin{enumerate}
\item We develop a precision-induced random re-normalization procedure that adaptively determines the $\epsilon$-rank $r_{\epsilon}$ and the low-rank approximation in a matrix-driven manner. Furthermore, we establish the necessary and sufficient condition for determining $r_{\epsilon}$, which is consistent with the prescribed approximation error bound and provides a solid theoretical foundation for the proposed framework.
\item Based on this procedure, we propose new SVD and matrix inverse algorithms for nearly low-rank matrices, where the iteration termination criterion differs from that of existing methods such as randQB-EI. The proposed stopping criterion leads to lower theoretical complexity and reduced computational cost while maintaining comparable accuracy.
\end{enumerate}

The rest of the paper is organized as follows. In Section 2, we review the Gaussian random re-normalization procedure. The corresponding theoretical analysis of showing its approximation performance is also given. We provide a precision-induced random re-normalization for accelerating the computation of the SVD of a nearly low-rank matrix in Section 3. We also theoretically justify the proposed methods by showing the error bound and computation complexity. 
We develop a precision‑induced random re‑normalization for fast matrix inversion and validate its efficacy through detailed error analysis in Section 4.
In Section 5, some numerical examples are proposed to illustrate the efficiency of the new methods. Finally, we conclude the paper in Section 6 with some useful remarks.

\subsection{Notations}
Let $\mathbb{R}^{m\times n}$, $\mathbb{C}^{m\times n}$, $\mathbb{C}^{n}$ and $\mathbb{U}_n$ denote the sets of $m\times n$ real matrices, $m\times n$ complex matrices, $n$-dimensional complex vectors and $n\times n$ unitary matrices, accordingly.
For a real matrix $A\in\mathbb{R}^{m\times n}$, by $A^{\mathrm{T}}$ we denote the transpose.
For a complex matrix $A\in\mathbb{C}^{m\times n}$, by $A^{\mathrm{H}}$, $A^{-1}$, $A^{\dag}$, $\mathrm{rank}(A)$ and $\mathrm{tr}(A)$ we denote the conjugate transpose, inverse, Moore-Penrose pseudo-inverse, rank and trace of matrix $A$, respectively.
$\sigma_{\max}(A)$ denotes the largest singular value of matrix $A$.
For matrices $A, B\in\mathbb{C}^{n\times n}$,  $A\succeq B$ means that $A-B$ is Hermitian and positive semi-definite. 
The symbols $I_{n}$ and $O_{m\times n}$ represent the identity matrix of order $n$ and $m\times n$ zero matrix, respectively.
We use $\|\cdot\|,\|\cdot\|_{2}$ and $\|\cdot\|_{\mathrm{F}}$ to denote the unitary invariant norm, spectral norm and Frobenius norm of a matrix, respectively.
For a given matrix $T$, we write $P_{T}$ for the unique orthogonal projector with $\mbox{range}(P_{T})=\mbox{range}(T)$. When $T$ has full column rank, we can express this projector explicitly $P_{T} = T(T^{\mathrm{H}}T)^{-1}T^{\mathrm{H}}$. 

\section{Random re-normalization}
In this section, we state and formalize the random re-normalization. To ease the presentation, we first provide the definition of the random re-normalization matrix as follows.

\begin{definition}[Random re-normalization matrix]\label{x1}
A matrix $\Omega_k\in\mathbb{C}^{n\times k}$ is defined as random re-normalization matrix of a matrix $A\in\mathbb{C}^{m\times n}$ if $\Omega_k$ satisfies random re-normalization condition:
\begin{equation}\label{Eq.cond}
  \mathrm{rank}(A\Omega_k)=\mathrm{rank}(A).
\end{equation}
\end{definition}

Definition \ref{x1} shows that, the rank of recombination matrix $A\Omega_k$ is that of $A$. When $k < \mbox{rank}(A)$, $\mbox{rank}(A\Omega_k)$ is always less than $\mbox{rank}(A)$ ; when $k \geq \mbox{rank}(A)$, $\mbox{rank}(A\Omega_k)=\mbox{rank}(A)$. Such changes of rank of $A\Omega_k$ makes it possible to automatically detect $\mbox{rank}(A)$ without comparing linear-relation of all columns. Numerically, columns of $A\Omega_k$ can be generated and checked their linear-relation in a sequential manner until $k \geq \mbox{rank}(A)$, which provides a doable way to simplify calculations. Theoretically, when $k \geq \mbox{rank}(A)$, the range of $A$ and $A\Omega$ share the same basis; we may find the basis of $A$ from $k$ recombination columns of $A\Omega_k$ instead of $n$ columns of $A$.

\begin{rem}
  The definition of random re-normalization can also be extended to nearly low-rank case. Specifically, for a nearly low-rank matrix $A\in\mathbb{C}^{m\times n}$ and precision $\epsilon > 0$, a matrix $\Omega_k \in\mathbb{C}^{n\times k}$ is a random re-normalization matrix of $A$ for $\epsilon$ if $\mathrm{rank}(A\Omega_k) \geq r_{\epsilon}$, where $r_{\epsilon}$ is the $\epsilon$-rank.
\end{rem}

The procedure of random re-normalization is iterative in nature and its basic steps are as follows: we initialize the procedure with a random column vector $\omega_1$ and compute $A\omega_1$; then we generate another random column vector $\omega_2$ and check linear-relation of $\{A\omega_1, A\omega_2\}$. This process is repeated until $\{A\omega_1, A\omega_2, \ldots, A\omega_k\}$ are linear dependent. Finally, the orthogonal basis of $A$ is obtained by orthogonaling recombination columns $\{A\omega_1, A\omega_2, \ldots, A\omega_k\}$ based on Gram-Schmidt orthogonalization. We summarize this procedure in Algorithm \ref{Alg.RR} below.

\begin{algorithm}
 \renewcommand{\algorithmicrequire}{\textbf{Input:}}
 \renewcommand{\algorithmicensure}{\textbf{Output:}}
 \caption{The random re-normalization procedure}
 \label{Alg.RR}
 \begin{algorithmic}[1]
  \REQUIRE A matrix $A\in\mathbb{C}^{m\times n}$.
  \ENSURE An approximate basis $Q=[q_{1},\ldots,q_{k}]$ of the range of $A$ and $\rank(A)$.
  \FOR{$k=1,2,3,\ldots$}
        \STATE Generate a random  vector $\omega_{k}=[\omega_{1k},\omega_{2k},\ldots,\omega_{nk}]^{\mathrm{T}}\in\mathbb{R}^{n}$ and compute $v_{k}=A\omega_{k}$.
        \IF{$\{v_1,\ldots, v_k\}$ are linear independent}
        \STATE Compute the orthogonal vectors $\{q_1,\ldots, q_k\}$ of $\{v_1,\ldots, v_k\}$.
        \ELSE
        \STATE \textbf{break}
        \ENDIF
        \ENDFOR
  \STATE \textbf{return} $Q=[q_{1},\ldots,q_{k}]$ and $\mbox{rank}(A)=k$.
 \end{algorithmic}
\end{algorithm}

\subsection{Gaussian random re-normalization}\label{Subsection.Gau}
Commonly, any Gaussian matrix $\Omega\in\mathbb{C}^{n\times k}$ can be served as  a random re-normalization matrix of a given matrix $A$ when $k \geq \mbox{rank}(A)$. 

\begin{lemma}\label{Thm.Gaussian}
Given a matrix $A\in\mathbb{C}^{m\times n}$, any Gaussian random matrix $\Omega\in\mathbb{C}^{n\times k}$ for all $\mathrm{rank}(A) \leq k\leq n$ is random re-normalization matrix of matrix $A$.
\end{lemma}
\if1\blind{
\begin{proof}
Let the SVD of $A\in\mathbb{C}^{m\times n}$ be $A=U\Sigma V^{\mathrm{H}}$, where $\mbox{rank}(A)=r$, $U\in \mathbb{U}_{m}$, $\Sigma\in\mathbb{R}^{m \times n}$ and $V\in \mathbb{U}_{n}$. Then, we have
\[
A\Omega=U\Sigma V^{\mathrm{H}}\Omega\equiv\left[
 \begin{array}{cc}
 U_1 & U_2 \\
 \end{array}
 \right]
 \left[
 \begin{array}{cc}
 \Sigma_1 & O_{r\times(n-r)} \\
  O_{(m-r)\times r}& O_{(m-r)\times(n-r)}\\
 \end{array}
 \right]
  \left[
  \begin{array}{c}
  V_1^{\mathrm{H}} \\
  V_2^{\mathrm{H}} \\
  \end{array}
 \right]
\Omega=U_1\Sigma_1V_1^{\mathrm{H}}\Omega,
\]
where $U_1 \in \mathbb{C}^{m\times r}, \Sigma_1=\diag(\sigma_1(A),\ldots,\sigma_r(A))\in \mathbb{R}^{r \times r}, V_1\in \mathbb{C}^{n \times r}$. By the unitary invariance of Gaussian matrices, we know that $V_1^{\mathrm{H}}\Omega$ is a $r\times k$ Gaussian random matrix and $\mbox{rank}(V_1^{\mathrm{H}}\Omega)=r$ (with probability one).
Since $U_1\Sigma_1\in \mathbb{C}^{m\times r}$ has full column rank, Lemma \ref{p2}  implies that
\begin{align*}\label{Eq.rank}
  \mbox{rank}(A\Omega)=\mbox{rank}(U_1\Sigma_1V_1^{\mathrm{H}}\Omega)&=\mbox{rank}(V_1^{\mathrm{H}}\Omega)=r=\mbox{rank}(A).
\end{align*}
\end{proof}
} \fi
\if0\blind{
\bigskip
}\fi

Lemma \ref{Thm.Gaussian} indicates that practitioners may use Gaussian distribution to directly construct a Gaussian random re-normalization matrix in practice. Once $\Omega_k$ is provided, instead of $A$, we can determine the approximate basis matrix $Q$ of $A$ directly from $A\Omega_k$. However, it is clear that Steps 3--5 in Algorithm \ref{Alg.RR} are computationally demanding since linearly independence is usually checked by Gram-Schmidt orthogonalization with $O(mk^2)$ complexity. The total computational cost of Algorithm \ref{Alg.RR} is more consuming especially as Gram-Schmidt orthogonalization is implemented in each iteration.

To improve the efficiency, we achieve Steps 3--5 of Algorithm \ref{Alg.RR} through QB decomposition \cite{10}. The QB decomposition is a basic tool of complexity reduction of matrix computation, which aims to decompose a matrix $A \in \mathbb{C}^{m \times n}$ into the form $A =QB$, where $Q \in \mathbb{C}^{m \times r}$ has orthonornal columns and $B\in \mathbb{C}^{r \times n}$ is an upper triangular matrix with $r=\mbox{rank}(A)$. The QB decomposition is an $O(mnr)$ complexity procedure, which can also be conveniently achieved in a linear manner. Namely, the columns of $A$ are added into the decomposition sequentially to obtain the colunmns of the upper triangular matrix $B$. In the random re-normalization process, this can be reflected by adding the columns of $A\Omega_k$ sequentially to obtain $A\Omega_k=Q_kB_k$. Using the QB decomposition of $A\Omega_k$, not $A$, may bring an exclusive advantage: only first $\mbox{rank}(A)$ columns of $A\Omega_k$ are independent and this can be reflected by diagonal entries of upper triangular expression $B_k$. Namely, the diagonal entries $b_{kk}$ of $B_k$ are all are nonzero whenever $k \leq \mbox{rank}(A)$, which thus can be used to terminate the random re-normalization process. The method is as follows. We recursively define $\Omega_{k+1} = [\Omega_k, \omega_{k+1}]$  and implement QB decomposition on $A\Omega_{k+1}$, resulting in QB decompositions $A\Omega_{k+1} = Q_{k+1}B_{k+1}$. Terminate the process when an index $k$ appears such that $b_{kk}$ is non-zero but $b_{k+1,k+1}$ does. Obviously, $\mbox{rank}(A)$ can be identified in this way and, more exclusively, this is realized in the QB decomposition process of $A\Omega_{k+1}$ with no additional computation involved. We summarize the Gaussian random re-normalization in the following Algorithm \ref{Alg.Gau}.
\begin{algorithm}
 \renewcommand{\algorithmicrequire}{\textbf{Input:}}
 \renewcommand{\algorithmicensure}{\textbf{Output:}}
 \caption{Gaussian random re-normalization}
 \label{Alg.Gau}
 \begin{algorithmic}[1]
  \REQUIRE A  matrix $A\in\mathbb{C}^{m\times n}$ and $q_{0}=[]$.
  \ENSURE A Gaussian random matrix $\Omega_{k}$ and an approximate  basis $Q_k=[q_{1},\ldots,q_{k}]$ of the range of $A$.

  \FOR{$k=1,2,\ldots,n$}
        \STATE Generate a Gaussian test vector $\omega_{k}=[\omega_{1k},\omega_{2k},\ldots,\omega_{nk}]^{\mathrm{T}}\in\mathbb{R}^n$ with $\omega_{jk} \sim \mathcal{N}(0,1)$ for $j=1, \ldots, n$.
        \STATE Compute $v_{k}=(I-\sum_{l=0}^{k-1}q_{l}q_{l}^{\mathrm{H}})A\omega_{k}$ and $b_{kk}=\|v_{k}\|_{2}$.
        \IF{$b_{kk}^{2}>0$}
        \STATE Compute $q_{k}=v_{k}/b_{kk}$.
        \ELSE
        \STATE \textbf{break}
        \ENDIF
        \ENDFOR
  \STATE \textbf{return} $\Omega_{k}=[\omega_{1},\omega_{2},\ldots,\omega_{k}]$ and $Q_k=[q_1,q_2,\ldots, q_{k}]$.
 \end{algorithmic}
\end{algorithm}

\subsection{Adaptive determination of $\epsilon$-rank}\label{Sec3.1}
When matrix $A$ is of nearly low-rank, by Definition \ref{Ass1}, only $r_{\epsilon}$ out of all singular values are significant for a precision $\epsilon$; the low-rank approximation of $A$ can be thus obtained with all insignificant singular values being disregarded.

However, $r_{\epsilon}$ is usually unknown in advance for many applications. Random re-normalization provides a viable strategy to determine $r_{\epsilon}$, due to its genuine existence of incremental update of columns of $\Omega$. That is, $r_{\epsilon}$ and $Q_k$ can be obtained simultaneously without any additional computation involved. Specifically, the columns of $A\Omega_k$ always are linearly independent when $k<r_{\epsilon}$. This can be reflected on $b_{kk}$ of the upper triangular matrix $B_k$, i.e., its diagonal entries are always relative large when index $k \leq r_{\epsilon}$, as discussed in Section \ref{Subsection.Gau}. Thus, we can employ QB decomposition on $A\Omega_k$ to determine $r_{\epsilon}$ instead of directly calculating all singular values of $A$. The basic steps are similarly to Algorithm \ref{Alg.Gau}. However, the columns of $A\Omega_k$ are weak-dependent even when $k>r_{\epsilon}$ in nearly low-rank case. It is thus important to first check when the process should be terminated as $b_{kk}$ decreasing. In Lemma \ref{Thm.Expect} below, we provide a necessary and sufficient condition of $b_{kk}$ and show the effectiveness of the low-rank approximation. 

\begin{lemma}\label{Thm.Expect}
Let Algorithm {\rm \ref{Alg.Gau}} be implemented on a matrix $A\in\mathbb{C}^{m\times n}$. Then, for any $0\leq \epsilon< 1$, the expectation of $b_{k+1,k+1}^2$ satisfies $\mathbb{E} b_{k+1,k+1}^2\leq\| A\|_{\mathrm{F}}^{2}\epsilon$ if and only if
$
\| A-\hat{A}\|_{\mathrm{F}}/\| A\|_{\mathrm{F}} \leq \sqrt{\epsilon}, 
$
where $\hat{A}=Q_{k}Q_{k}^{\mathrm{H}}A$. In this case, we have $(\sum_{i=1}^{k}\sigma_{i}^{2}(A))/(\sum_{i=1}^{p}\sigma_{i}^{2}(A))$ $\geq 1-\epsilon$ and thus the $\epsilon$-rank of $A$ is $r_{\epsilon}=k$, where $p=\min\{m,n\}$.
\end{lemma}

Lemma \ref{Thm.Expect} implies that, with an appropriate condition of  $\mathbb{E} b_{k+1,k+1}^2$, the proposed random re-normalization procedures automatically determine the $r_{\epsilon}$ while retaining a low-rank approximation. Although Lemma \ref{Thm.Expect} justifies the effectiveness of the proposed procedures, it is computationally infeasible in practice. For practical use, a doable termination condition of $b_{k+1,k+1}$ is more appealing. Lemma \ref{vq0} provides such a result.
\begin{lemma}\label{vq0}
Let Algorithm {\rm \ref{Alg.Gau}} be implemented on a matrix $A\in\mathbb{C}^{m\times n}$. For any $0\leq \epsilon< 1$, if Algorithm {\rm\ref{Alg.Gau}} terminates by $b_{k+1,k+1}^2\leq\| A\|_{\mathrm{F}}^{2}\epsilon/2$, then 
\begin{equation}\label{result1}
\| A-\hat{A}\|_{\mathrm{F}}/\| A\|_{\mathrm{F}} \leq \sqrt{\epsilon}, \quad \mbox{$ {(\sum_{i=1}^{k}\sigma_{i}^{2}(A))}/{(\sum_{i=1}^{p}\sigma_{i}^{2}(A))}\geq1-\epsilon$},\quad r_{\epsilon}=k
\end{equation}
hold with probability at least $1-2 \exp(-c\epsilon^2 \sum_{i=1}^{p}\sigma_i^2(A) / (4K^4\sum_{i=1}^{p-k} \sigma_i^2(A)))$, where $\hat{A}=Q_{k}Q_{k}^{\mathrm{H}}A$, $p=\min\{m,n\}$, and $c,K>0$ are  constants.
\end{lemma}

Lemma \ref{vq0} implies that, with an termination condition of $b_{k+1,k+1}^2$, Algorithm \ref{Alg.Gau} can automatically determine the nearly low-rank $r_{\epsilon}$ and yield the corresponding low-rank approximation of $A$ with high probability. This lemma provides a practical guidance for termination of Algorithm \ref{Alg.Gau}.
Based on the decision rule of Lemma \ref{vq0}, we readily derive Algorithm \ref{alg:qr} by reformulating Algorithm \ref{Alg.Gau} into a block‐based computation framework, thereby significantly enhancing its computational efficiency.

\begin{algorithm}[H]
    \caption{Precision-induced block Gaussian random re-normalization algorithm}
    \label{alg:qr}
    \renewcommand{\algorithmicrequire}{\textbf{Input:}}
    \renewcommand{\algorithmicensure}{\textbf{Output:}}
    \begin{algorithmic}[1]
    \REQUIRE A nearly low-rank matrix $A\in\mathbb{C}^{m\times n}$, a tolerance $0\le \epsilon<1$, and a blocksize $1\leq b<n$. Let $s = \lfloor n/b\rfloor$ and $f_{j} = b$, $j=1,2,\ldots,s$, $f_{s+1} = n-sb$.  
    \ENSURE An approximate basis $Q\in\mathbb{C}^{m\times k}$ of the range of $A$.   

\STATE $Q = [\;]$.

\FOR {$j=1,2,\ldots,s+1$} \label{for}

\STATE Draw a standard Gaussian matrix $\Omega_{j}\in\mathbb{R}^{n\times f_{j}}$.

\STATE Set $Y_{j} = (I_{m}-QQ^{\mathrm{H}})A\Omega_{j}$.

\STATE Compute the QR decomposition $Y_{j} = P_{j}T_{j}$, where $P_{j} \in \mathbb{C}^{m\times f_j}$, $T_{j} \in \mathbb{C}^{f_j\times f_j}$.

\IF{find the smallest $1\leq \ell\leq f_j$ with $|T_{j}(\ell,\ell)|\leq \|A\|_{\mathrm{F}}\sqrt{\epsilon}/\sqrt{2}$} \label{stop}

\STATE $Q=[Q,P_{j}(:,1:\ell-1)]$.

\STATE \textbf{break}

\ELSE

\STATE $Q=[Q,P_{j}]$.

\ENDIF

\ENDFOR

        \RETURN $Q$. \label{return}
    \end{algorithmic}
\end{algorithm}

\begin{rem}
Assuming that multiplying two dense matrices of sizes \(m\times n\) and \(n\times k\) costs \(C_{\mathrm{mm}}mnk\) flops, and performing an economic QR decomposition of an \(m\times n\) dense matrix costs \(C_{\mathrm{qr}}mn\min(m,n)\) flops, the computational complexity of the randQB-EI algorithm in \cite{yu2018} is approximately
\[
2C_{\mathrm{mm}}mnk + C_{\mathrm{mm}}(2m + n)k^{2} + \tfrac{2}{s}C_{\mathrm{qr}}mk^{2},
\]
where \(k\) denotes the target rank.
In contrast, the proposed Algorithm~\ref{alg:qr} has a computational complexity of about
\[
C_{\mathrm{mm}}mnk + 2C_{\mathrm{mm}}mk^{2} + \tfrac{1}{s}C_{\mathrm{qr}}mk^{2}.
\]
Accordingly, our method reduces the computational cost by roughly
\[
C_{\mathrm{mm}}mnk + C_{\mathrm{mm}}nk^{2} + \tfrac{1}{s}C_{\mathrm{qr}}mk^{2}.
\]
This improvement mainly stems from two aspects: (i) a different termination criterion, which eliminates the need to explicitly form the matrix \(B\)-the stopping condition is determined directly from the diagonal entries of the R-factor obtained in Step~5 of the QR decomposition; and (ii) the absence of any reorthogonalization procedure, which further reduces the overall computational cost.
\end{rem}

\section{Accelerating SVD of nearly low-rank matrix}
We now turn to apply random re-normalization to simplify the computation of SVD. In particular, we aim to answer two key questions: 1) how to implicitly embed the random re-normalization to the computation of SVD; 2) how the performance of the new method.

As revealed in Section \ref{Sec3.1}, the $\epsilon$-rank $r_{\epsilon}$ of $A$ and an approximation basis $Q_{k}$ can be determined by Algorithm \ref{alg:qr}. Thus, the column space of $A$ can be embedded into a low-dimensional space without altering its geometric properties substantially. Instead of directly finding SVD of $A$ from the column space, we can intuitively compute SVD from the low-dimensional space spanned by $Q_k$. Specifically, we first project $A$ into $r_{\epsilon} \times n$ matrix $Q_{k}^{\mathrm{H}}A$. In this way, the proposed procedures transfer the SVD of an $m\times n$ matrix $A$ into the eigendecomposition of a $r_{\epsilon}\times r_{\epsilon}$ matrix $Q_{k}^{\mathrm{H}}AA^{\mathrm{H}}Q_{k}$, i.e., $Q_{k}^{\mathrm{H}}AA^{\mathrm{H}}Q_{k}=U\Sigma_{1} U^\mathrm{H}$, which significantly saves the computation resources. Then the left singular vector matrix $\hat{U}$ of $A$ can be deduced by $Q_{k}U$, the singular value matrix $\hat{\Sigma}$ can be deduced by $\Sigma_{1}^{\frac{1}{2}}$, and the right singular vector matrix $\hat{V}^\mathrm{H}$ can be obtained by $\Sigma_{1}^{-\frac{1}{2}} U^{\mathrm{H}} Q_{k}^{\mathrm{H}}A$. The proposed procedure is summarized in Algorithm \ref{Alg.GBRSVD}.

\begin{algorithm}[H]
 \renewcommand{\algorithmicrequire}{\textbf{Input:}}
 \renewcommand{\algorithmicensure}{\textbf{Output:}}
 \caption{Gaussian random re-normalization SVD (GRSVD) for a nearly low-rank matrix}
 \label{Alg.GBRSVD}
 \begin{algorithmic}[1]
  \REQUIRE A nearly low-rank matrix $A\in\mathbb{C}^{m\times n}$ and a tolerance $0\le \epsilon<1$.
  \ENSURE  An approximate SVD of  $A$: $A\approx \hat{U}\hat{\Sigma} \hat{V}^\mathrm{H}$.
        \STATE Run Algorithm \ref{alg:qr} and output $Q_{k}=Q$.
        \STATE Set $B=Q_{k}^{\mathrm{H}}A$ and compute $C=BB^{\mathrm{H}}$.
        \STATE Compute the eigendecomposition $C=U\Sigma_{1} U^\mathrm{H}$.
        \STATE Compute $\hat{U}=Q_{k}U,\hat{\Sigma}=\Sigma_{1}^{1/2},\hat{V}^\mathrm{H}=\Sigma_{1}^{-1/2} U^\mathrm{H} B$.
 \end{algorithmic}
\end{algorithm}

The algorithm first applies Algorithm~\ref{alg:qr} to compute the rank-$k$ orthonormal factor $Q_k$, which costs $O(mnk)$ flops. Then $A$ is projected onto the subspace spanned by $Q_k$ by forming $B = Q_k^{\mathrm{H}} A$ and computing $C = B B^{\mathrm{H}}$, with computational cost $mnk + nk^2$. Next, the eigendecomposition $C = U \Sigma_1 U^{\mathrm{H}}$ is computed, requiring $O(k^3)$ flops. Finally, the approximate singular vectors and singular values are reconstructed by evaluating $\hat{U} = Q_k U$, $\hat{\Sigma} = \Sigma_1^{\frac{1}{2}}$, and $\hat{V}^{\mathrm{H}} = \Sigma_1^{-\frac{1}{2}} U^{\mathrm{H}} B$, which involves $mk^2 + nk^2 + nk + k$ flops. Among these steps, the computation is dominated by the initial QR decomposition of $A$, so the overall complexity is approximately $mnk$ when $m, n \gg k$.

Now, we provide the error analysis of singular values between $A$ and $\hat{A}$ for GRSVD in below Theorem \ref{zxr1}.
\begin{theorem}\label{zxr1}
Let Algorithm {\rm\ref{Alg.GBRSVD}} be applied to a matrix $A\in\mathbb{C}^{m\times n}$. Then 
\begin{eqnarray*}
&&\| A-\hat{A}\|_{\mathrm{F}}/\| A\|_{\mathrm{F}} \leq \sqrt{\epsilon},\\ 
&&\mbox{$0\le\big(\sum_{i=1}^{k}\sigma_{i}^2(A)-\sum_{i=1}^{k}\sigma_{i}^2(\hat{A})\big)/ \sum_{i=1}^{k}\sigma_{i}^2(A)\leq \Vert A\Vert_\mathrm{F}^2/ \sum_{i=1}^{k}\sigma_{i}^2(A)\epsilon$}, \\
&& \sigma_{i}^2(A)-\Vert A\Vert_\mathrm{F}^2\epsilon\le\sigma_{i}^2(\hat{A})\leq \sigma_{i}^2(A), \quad i=1,\ldots,k
\end{eqnarray*}
hold with probability at least $1-2 \exp(-c\epsilon^2 \sum_{i=1}^{p}\sigma_i^2(A) / (4K^4\sum_{i=1}^{p-k} \sigma_i^2(A)))$, where $\hat{A}=\hat{U}\hat{\Sigma} \hat{V}^\mathrm{H}$, $p=\min\{m,n\}$, and $c,K>0$ are  constants as defined in Lemma {\rm\ref{vq0}}.
\end{theorem}

\begin{proof}
By Algorithm {\rm\ref{Alg.GBRSVD}} we have
\begin{eqnarray*}
\hat{A} &= & \hat{U}\hat{\Sigma} \hat{V}^\mathrm{H}=Q_kU\Sigma_{1}^{1/2}\Sigma_{1}^{-1/2} U^\mathrm{H} Q_k^\mathrm{H}A=Q_{k}Q_{k}^{\mathrm{H}}A.
\end{eqnarray*}
Let $G=(I-Q_{k}Q_{k}^{\mathrm{H}})A$. Using Step 6 of Algorithm \ref{alg:qr} and Lemma \ref{vq0}, it follows that
\begin{equation}\label{Eq.thm32}
\| A-\hat{A}\|_{\mathrm{F}}^2 =\|G\|_\mathrm{F}^2\leq \| A\|_{\mathrm{F}}^2\epsilon 
\end{equation}
holds with probability at least $1-2 \exp(-c\epsilon^2 \sum_{i=1}^{p}\sigma_i^2(A) / (4K^4\sum_{i=1}^{p-k} \sigma_i^2(A)))$.

From \eqref{ieq:aat-sg2} we have
\begin{align}\label{ieq.difsigma}
&\|A\|_\mathrm{F}^2 - \|\hat{A}\|_\mathrm{F}^2 = \sum_{i=1}^{p}\sigma_{i}^{2}(A) - \sum_{i=1}^{k}\sigma_{i}^{2}(\hat{A}) \nonumber\\
\geq &\sum_{i=1}^{k}(\sigma_{i}^{2}(A)-\sigma_{i}^{2}(\hat{A}))\geq\max_{1\leq i\leq k}(\sigma_{i}^{2}(A)-\sigma_{i}^{2}(\hat{A}))\ge 0.
\end{align}
\eqref{Eq.thm32}, together with \eqref{ieq.difsigma} and \eqref{ieq:aat-sg2}, implies that
\begin{eqnarray*}
&&\mbox{$0\le\big(\sum_{i=1}^{k}\sigma_{i}^2(A)-\sum_{i=1}^{k}\sigma_{i}^2(\hat{A})\big)
/\sum_{i=1}^{k}\sigma_{i}^2(A)\leq \Vert A\Vert_\mathrm{F}^2/\sum_{i=1}^{k}\sigma_{i}^2(A)\epsilon$}, \\
&&\sigma_{i}^2(A)-\Vert A\Vert_\mathrm{F}^2\epsilon\le\sigma_{i}^2(\hat{A})\leq \sigma_{i}^2(A),\;
i=1,\ldots,k
\end{eqnarray*}
holds with probability at least $1-2 \exp(-c\epsilon^2 \sum_{i=1}^{p}\sigma_i^2(A) / (4K^4\sum_{i=1}^{p-k} \sigma_i^2(A)))$.
\end{proof}

\begin{rem}
From \eqref{Eq.thm32}, \eqref{ieq.difsigma}, and \eqref{ieq:aat-sg2} we have
\[
0\le\frac{\sigma_{i}(A)-\sigma_{i}(\hat{A})}{\sigma_{i}(A)} \le \frac{\sqrt{\sigma_{i}^2(A)-\sigma_{i}^2(\hat{A})}}{\sigma_{i}(A)}\leq \frac{\|A\|_\mathrm{F}}{\sigma_{i}(A)}\sqrt{\epsilon},\quad i=1,\ldots,k.
\]
This, together with Theorem {\rm\ref{zxr1}}, indicates that, given the energy ratio $0\le \epsilon<1$, the relative errors of singular values between $A$ and its approximate matrix $\hat{A}$ are the magnitude of $O(\sqrt{\epsilon})$ with high probability.
\end{rem}

\section{Accelerating matrix inversion of two matrix inverses}
We can further use random re-normalization to accelerate the inversion computation of $\lambda I_{m}+AA^{H}$ and $\lambda I_{n}+A^{H}A$ for arbitrary $\lambda>0$. Instead of directly computing the matrix inverses, the low-rank approximation sheds light on transferring the origin matrix  inversion problem of size $n \times n$ into a matrix inversion problem of size $r_{\epsilon} \times r_{\epsilon}$, where $r_{\epsilon}\ll \min\{m,n\}$.
With the Sherman-Morrison-Woodbury formula  in mind, we provide the computation techniques of two matrix inverses.

\begin{theorem}\label{l2}
Let $A\in\mathbb{C}^{m\times n}$ and $\lambda>0$. Suppose $Q_{k}=Q$ is computed by using Algorithm {\rm\ref{alg:qr}} to $A$. Then we have 
\begin{eqnarray*}
&& (\lambda I_{m}+AA^{\mathrm{H}})^{-1}\approx(\lambda I_{m}+\hat{A}\hat{A}^{\mathrm{H}})^{-1}=\frac{1}{\lambda}I_{m}+Q_{k}\Big((\lambda I_{k}+BB^{\mathrm{H}})^{-1}-\frac{1}{\lambda}I_{k}\Big)Q_{k}^{\mathrm{H}}, \\
&& (\lambda I_{n}+A^{\mathrm{H}}A)^{-1}\approx(\lambda I_{n}+\hat{A}^{\mathrm{H}}\hat{A})^{-1}=\frac{1}{\lambda}I_{n}-\frac{1}{\lambda^{2}}B^{\mathrm{H}}(I_{k}+\frac{1}{\lambda}BB^{\mathrm{H}})^{-1}B,
\end{eqnarray*}
where $\hat{A}=Q_{k}Q_{k}^{\mathrm{H}}A$ and $B=Q_{k}^{\mathrm{H}}A$.
\end{theorem}

Theorem \ref{l2} shows that, the inverses of $\lambda I_{m}+AA^{\mathrm{H}}$ and $\lambda I_{n}+A^{\mathrm{H}}A$ can be approximated by the inverses of $\lambda I_{m}+\hat{A}\hat{A}^{\mathrm{H}}$ and $\lambda I_{n}+\hat{A}^{\mathrm{H}}\hat{A}$, respectively, with $B=Q_{k}^{\mathrm{H}}A.$ Since $B$ is of size $k \times n$, Theorem \ref{l2} indicates that we can compute the inverses of  $\lambda I_{k}+BB^{\mathrm{H}}$  and $I_{k}+{1}/{\lambda}BB^{\mathrm{H}}$ by their Cholesky decompositions. The total computation complexity thus reduces from $O(n^3)$ to $O(n^2k)$. We first project $A$ into $k \times n\;Q_{k}^{\mathrm{H}}A.$ In this way, the proposed procedures transfer an $m\times m$ inversion problem of $\lambda I_{m}+AA^{\mathrm{H}}$ into a $k\times k$ inversion problem of $\lambda I_{k}+BB^{\mathrm{H}}$. By computing the Cholesky decomposition $\lambda I_{k}+BB^{\mathrm{H}}=LL^{\mathrm{H}}$, and then deduce $\lambda (I_{m}+\hat{A}\hat{A}^{\mathrm{H}})^{-1}={1}/{\lambda}I_{m}+Q_{k}(L^{-\mathrm{H}}L^{-1}-{1}/{\lambda}I_{k})Q_{k}^{\mathrm{H}}$. This significantly saves the computation resources. Similarly, the proposed procedures transfer an $n\times n$ inversion problem of $\lambda I_{n}+A^{\mathrm{H}}A$ into a $k\times k$ inversion problem of $I_{k}+{1}/{\lambda}BB^{\mathrm{H}}$. By computing the Cholesky decomposition  $I_{k}+{1}/{\lambda}BB^{\mathrm{H}}=LL^{\mathrm{H}}$, and then deduce $(\lambda I_{n}+\hat{A}^\mathrm{H}\hat{A})^{-1}={1}/{\lambda}I_{n}-{1}/{\lambda^{2}}B^{\mathrm{H}}(L^{-\mathrm{H}}L^{-1})B$. This significantly also saves the computation resources.

We summarize the procedures of Gaussian random re-normalization for two matrix inverses in the following Algorithms \ref{Alg.GBRI} and \ref{Alg.GBRRI}, respectively.

\begin{algorithm}[H]
 \renewcommand{\algorithmicrequire}{\textbf{Input:}}
 \renewcommand{\algorithmicensure}{\textbf{Output:}}
 \caption{Gaussian random re-normalization inversion (GRI) for $\lambda I_{m}+AA^{\mathrm{H}}$}
 \label{Alg.GBRI}
 \begin{algorithmic}[1]
  \REQUIRE A nearly low-rank matrix $A\in\mathbb{C}^{m\times n}$ and a number $\lambda>0$.
  \ENSURE  Approximate inverse of $\lambda I_{m}+AA^{\mathrm{H}}$.
        \STATE Run Algorithm \ref{alg:qr} and output $Q_{k}=Q$.
        \STATE Set $B=Q_{k}^{H}A$ and compute the Cholesky decomposition $\lambda I_{k}+BB^{\mathrm{H}}=LL^{\mathrm{H}}$.
        \STATE Compute $(\lambda I_{m}+\hat{A}\hat{A}^{\mathrm{H}})^{-1}=1/\lambda I_{m}+Q_{k}(L^{-\mathrm{H}}L^{-1}-1/\lambda I_{k})Q_{k}^{\mathrm{H}}$.
 \end{algorithmic}
\end{algorithm}

\begin{algorithm}[H]
 \renewcommand{\algorithmicrequire}{\textbf{Input:}}
 \renewcommand{\algorithmicensure}{\textbf{Output:}}
 \caption{Gaussian random re-normalization inversion (GRI) for $\lambda I_{n}+A^{\mathrm{H}}A$}
 \label{Alg.GBRRI}
 \begin{algorithmic}[1]
  \REQUIRE A nearly low-rank matrix $A\in\mathbb{C}^{m\times n}$  and a number $\lambda>0$.
  \ENSURE  Approximate inverse of $\lambda I_{n}+A^{\mathrm{H}}A$.
        \STATE Run Algorithm \ref{alg:qr} and output $Q_{k}=Q$.
        \STATE Set $B=Q_{k}^{\mathrm{H}}A$ and compute the Cholesky decomposition $I_{k}+1/{\lambda}BB^{\mathrm{H}}=LL^{\mathrm{H}}$.
        \STATE Compute $(\lambda I_{n}+\hat{A}^\mathrm{H}\hat{A})^{-1}={1}/{\lambda}I_{n}-{1}/{\lambda^{2}}B^{\mathrm{H}}(L^{-\mathrm{H}}L^{-1})B$.
 \end{algorithmic}
\end{algorithm}

In Algorithm \ref{Alg.GBRI}, the computation of the rank-$k$ orthonormal basis $Q_k$ using Algorithm \ref{alg:qr} costs approximately $O(mnk)$. Then, the projection $B = Q_k^{\mathrm{H}}A$ and the Cholesky decomposition of $\lambda I_k + BB^{\mathrm{H}}$ together incur a cost of about $mnk + nk^2 + k^3/3$. Finally, forming $(\lambda I_m + \hat{A}\hat{A}^{\mathrm{H}})^{-1}$ involves matrix multiplications and inversion with a total cost of roughly $k^3/3 + mk^2 + m^2k$.
In Algorithm \ref{Alg.GBRRI}, the computation of $Q_k$ similarly costs approximately $O(mnk)$. The subsequent projection $B = Q_k^{\mathrm{H}}A$ and the Cholesky decomposition of $I_k + {1}/{\lambda}BB^{\mathrm{H}}$ have a cost of about $mnk + nk^2 + k^3/3$. Finally, computing $(\lambda I_n + \hat{A}^{\mathrm{H}}\hat{A})^{-1}$ requires about $k^3/3 + nk^2 + n^2k$.
In both algorithms, the leading computational cost comes from the initial QR factorization and the projection step, while the cubic‑in‑$k$ terms are relatively small when $k \ll \min\{m,n\}$.

We now turn to evaluate the inversion performance of new methods. To this end, we investigate the Frobenius norm of $(\lambda I_{m}+AA^{\mathrm{H}})^{-1}-(\lambda I_{m}+\hat{A}\hat{A}^{\mathrm{H}})^{-1}$ and $(\lambda I_{n}+A^{\mathrm{H}}A)^{-1}-(\lambda I_{m}+\hat{A}^{\mathrm{H}}\hat{A})^{-1}$ . We justify the inversion effectiveness of the new method in the following theorem.


\begin{theorem}\label{xx2}
Let $A\in\mathbb{C}^{m\times n}$ and $\lambda>0$. 
Suppose Algorithms {\rm\ref{Alg.GBRI}} and {\rm\ref{Alg.GBRRI}} are implemented on $A$. Then
\begin{eqnarray*}
  &&\frac{\Vert (\lambda I_{m}+AA^{\mathrm{H}})^{-1}-(\lambda I_{m}+\hat{A}\hat{A}^{\mathrm{H}})^{-1}\Vert_{\mathrm{F}}}{\|(\lambda I_{m}+AA^{\mathrm{H}})^{-1}\|_\mathrm{F}}\leq\frac{2}{\lambda}\Vert A\Vert_{2} \Vert A\Vert_{\mathrm{F}} \sqrt{\epsilon},\\
  &&\frac{\Vert(\lambda I_{n}+A^{\mathrm{H}}A)^{-1}-(\lambda I_{n}+\hat{A}^{\mathrm{H}}\hat{A})^{-1}\Vert_{\mathrm{F}}}{\Vert(\lambda I_{n}+A^{\mathrm{H}}A)^{-1}\|}\leq\frac{1}{\lambda}\Vert A\Vert_{\mathrm{F}}^2\epsilon,
\end{eqnarray*}
hold with probability at least $1-2 \exp(-c\epsilon^2 \sum_{i=1}^{p}\sigma_i^2(A) / (4K^4\sum_{i=1}^{p-k} \sigma_i^2(A)))$, where $\hat{A}=Q_{k}Q_{k}^{\mathrm{H}}A$.
\end{theorem}
\if1\blind{
\begin{proof}
Let
$
S_{1}=(\lambda I_{m}+AA^{\mathrm{H}})^{-1}$ and $S_{2}=(\lambda I_{m}+\hat{A}\hat{A}^{\mathrm{H}})^{-1}=S_{1}+\Delta S.
$
Thus,
\[
(\lambda I_{m}+AA^{\mathrm{H}})S_{1}=I_m=(\lambda I_{m}+\hat{A}\hat{A}^{\mathrm{H}})S_{2}=(\lambda I_{m}+\hat{A}\hat{A}^{\mathrm{H}})(S_{1}+\Delta S).
\]
Let $\Delta A=\hat{A}\hat{A}^{\mathrm{H}}-AA^{\mathrm{H}}= Q_{k}Q_{k}^{\mathrm{H}}AA^{\mathrm{H}}Q_{k}Q_{k}^{\mathrm{H}}-AA^{\mathrm{H}}$. Then
\begin{eqnarray*}
&&I_{m}=(\lambda I_{m}+\hat{A}\hat{A}^{\mathrm{H}})(S_{1}+\Delta S)=(\lambda I_{m}+AA^{\mathrm{H}}+\Delta A)(S_{1}+\Delta S)\\
&&=(\lambda I_{m}+AA^{\mathrm{H}})S_{1}+(\lambda I_{m}+AA^{\mathrm{H}})\Delta S+\Delta A(S_{1}+\Delta S)\\
&&=I_{m}+(\lambda I_{m}+AA^{\mathrm{H}})\Delta S+\Delta AS_{2},
\end{eqnarray*}
which implies that
\[
\Delta S=-(\lambda I_{m}+AA^{\mathrm{H}})^{-1} \Delta A (\lambda I_{m}+\hat{A}\hat{A}^{\mathrm{H}})^{-1}.
\]
Then
\begin{equation}\label{eq:ds}
\Vert \Delta S\Vert_{\mathrm{F}}\leq\frac{\|(\lambda I_{m}+AA^{\mathrm{H}})^{-1}\|_\mathrm{F} \Vert\Delta A\Vert_{\mathrm{F}}}{\sigma_{m}(\lambda I_{m}+\hat{A}\hat{A}^{\mathrm{H}})}\leq\frac{1}{\lambda} \|(\lambda I_{m}+AA^{\mathrm{H}})^{-1}\|_\mathrm{F} \Vert\Delta A\Vert_{\mathrm{F}}.
\end{equation}
Using Step 6 of Algorithm \ref{alg:qr} and Lemma \ref{vq0}, it follows that
\[
\Vert A-Q_{k}Q_{k}^{\mathrm{H}}A\Vert_{\mathrm{F}}^2=\| A-\hat{A}\|_{\mathrm{F}}^2 \leq \| A\|_{\mathrm{F}}^2\epsilon 
\]
holds with probability at least $1-2 \exp(-c\epsilon^2 \sum_{i=1}^{p}\sigma_i^2(A) / (4K^4\sum_{i=1}^{p-k} \sigma_i^2(A)))$.
Thus,
\begin{eqnarray}\label{51}
\Vert\Delta A\Vert_{\mathrm{F}}&=&\Vert Q_{k}Q_{k}^{\mathrm{H}}AA^{\mathrm{H}}Q_{k}Q_{k}^{\mathrm{H}}-AA^{\mathrm{H}}\Vert_{\mathrm{F}}\nonumber\\
&\le &\Vert Q_{k}Q_{k}^{\mathrm{H}}AA^{\mathrm{H}}(I_m-Q_{k}Q_{k}^{\mathrm{H}})\Vert_{\mathrm{F}}+\Vert (I_m-Q_{k}Q_{k}^{\mathrm{H}})AA^{\mathrm{H}}\Vert_{\mathrm{F}}\nonumber\\
&\leq&2\Vert A\Vert_{2}\Vert A-Q_{k}Q_{k}^{\mathrm{H}}A\Vert_{\mathrm{F}}=2\Vert A\Vert_{2}\Vert A\Vert_{\mathrm{F}}\sqrt{\epsilon}
\end{eqnarray}
holds with probability at least $1-2 \exp(-c\epsilon^2 \sum_{i=1}^{p}\sigma_i^2(A) / (4K^4\sum_{i=1}^{p-k} \sigma_i^2(A)))$.
This, together with \eqref{eq:ds}, indicates that
\[
\frac{\Vert (\lambda I_{m}+AA^{\mathrm{H}})^{-1}-(\lambda I_{m}+\hat{A}\hat{A}^{\mathrm{H}})^{-1}\Vert_{\mathrm{F}}} {\|(\lambda I_{m}+AA^{\mathrm{H}})^{-1}\|_\mathrm{F}}=\frac{\|\Delta S\|_\mathrm{F}} {\|(\lambda I_{m}+AA^{\mathrm{H}})^{-1}\|_\mathrm{F}} \leq\frac{2}{\lambda}\Vert A\Vert_{2} \Vert A\Vert_{\mathrm{F}}\sqrt{\epsilon}
\]
holds with probability at least $1-2 \exp(-c\epsilon^2 \sum_{i=1}^{p}\sigma_i^2(A) / (4K^4\sum_{i=1}^{p-k} \sigma_i^2(A)))$.

Analogously, we can verify that
\[
\frac{\Vert (\lambda I_{n}+A^{\mathrm{H}}A)^{-1}-(\lambda I_{n}+\hat{A}^{\mathrm{H}}\hat{A})^{-1}\Vert_{\mathrm{F}}} {\Vert(\lambda I_{n}+A^{\mathrm{H}}A)^{-1}\|}\leq\frac{1}{\lambda}\Vert A\Vert_{\mathrm{F}}^2\epsilon
\]
holds with probability at least $1-2 \exp(-c\epsilon^2 \sum_{i=1}^{p}\sigma_i^2(A) / (4K^4\sum_{i=1}^{p-k} \sigma_i^2(A)))$.
\end{proof} 
} \fi
\if0\blind{
\bigskip
}\fi

Theorem \ref{xx2} shows that, during the random re-normalization of Algorithm \ref{alg:qr}, if there exist the smallest $1\le\ell\le f_j$ such that $|T_{j}(\ell,\ell)|\leq \|A\|_{\mathrm{F}}\sqrt{\epsilon}/\sqrt{2}$ for some $j\ge 1$, the matrix inverses $(\lambda I_{m}+AA^{\mathrm{H}})^{-1}$ and $(\lambda I_{n}+A^{\mathrm{H}}A)^{-1}$ can be efficiently approximated by the matrix inverses $(\lambda I_{m}+\hat{A}\hat{A}^{\mathrm{H}})^{-1}$ and $(\lambda I_{m}+\hat{A}^{\mathrm{H}}\hat{A})^{-1}$, respectively, with high probability.




\section{Numerical examples}\label{Sec.simu}
We demonstrate the proposed  random re-normalization methods via a series of simulated examples. We evaluate GRSVD in terms of computational accuracy and efficiency by comparing with the economy-sized SVD (eSVD) by MATLAB command \texttt{svd(A,'econ')},  RSVD in \cite{9}, ARRF in \cite{9}, and randQB-EI \cite{yu2018}. All the tests were carried out in {\tt MATLAB 2021a} running on a
workstation with a Intel Xeon CPU Gold 6134 at 3.20 GHz and 32 GB of RAM.
\begin{example}\label{exam}
We run GRSVD on a series of simulated matrices $A\in \mathbb{C}^{10000\times 8000}$.
The matrices are generated by
\[
A=U_A\Sigma_A V_A^{\mathrm{H}},
\]
where both $U_A\in\mathbb{C}^{10000\times 8000}$ and $V_A\in\mathbb{C}^{8000\times 8000}$ has orthonormal columns and $\Sigma_A \in\mathbb{C}^{8000\times 8000}$ is a diagonal matrix. The diagonal elements of matrix $\Sigma_A$ are generated by {\tt MATLAB 2021a}'s command 
\[
\texttt{diag(sort([rand(r,1);rand(8000-r,1)*$\epsilon$],'descend'))}.
\]
The parameter $r$ (ranging from 1000 to 4000) represents the index at which the singular values begin to decay sharply. We consider several nearly low-rank matrices with $\epsilon=10^{-8}$.
\end{example}

Let $\{\hat{\sigma_i}\}_{i=1}^{r_{\epsilon}}$ be the set of $r_{\epsilon}$ singular values obtained from the corresponding decompositions. The accuracy of each method is assessed by the average relative errors of matrix reconstruction and singular values:
\[E_{MSE} = \frac{\|A-\hat{U}\hat{\Sigma}\hat{V}^{\mathrm{H}}\|_\mathrm{F}}{\|A\|_\mathrm{F}},\;\;E_{\Sigma} = \max\limits_{1\leq i\leq r_\epsilon}\frac{|\sigma_{i}^{2}(A)-\hat{\sigma}_{i}^{2}|}{|\sigma_{i}^{2}(A)|}\]
and the orthogonality of singular vector matrices:
\[E_{U} = \frac{\|\hat{U}^{\mathrm{H}}\hat{U}-I\|_{\mathrm{F}}}{\sqrt{r_\epsilon}},\;\;E_{V} = \frac{\|\hat{V}^{\mathrm{H}}\hat{V}-I\|_{\mathrm{F}}}{\sqrt{r_\epsilon}}.\]
The average computational time (in seconds) is also reported as a measure of the efficiency. Each experiment is repeated 100 times, and the results are summarized in Table \ref{Tab.example-nlr} and Figure \ref{Fig:nlrt}.

\begin{table}
  \centering\scriptsize
  \renewcommand{\arraystretch}{1.2} 
  \caption{Error comparisons of SVD for nearly low-rank matrices}\label{Tab.example-nlr}
  \begin{tabular}{ccccc}
    \toprule[1.2pt]
     $r/r_\epsilon$ & Method  &  $E_{MSE}$ & $E_{\Sigma}$ & $\max\{E_{U},E_{V}\}$\\
\midrule											
 		\multirow{5}*{1000} 	&	GRSVD	&	2.76E-07	&	5.72E-10	&	2.45E-15	\\
			&	RSVD	&	3.27E-08	&	9.64E-12	&	3.30E-15	\\
			&	ARRF	&	1.44E-07	&	6.13E-10	&	2.18E-15	\\
			&	randQB-EI	&	3.04E-08	&	7.91E-12	&	2.40E-15	\\
			&	eSVD	&	2.60E-15	&	4.94E-15	&	3.76E-15	\\
													
		\multirow{5}*{2000} 	&	GRSVD	&	6.11E-07	&	1.26E-09	&	3.40E-15	\\
			&	RSVD	&	2.08E-08	&	1.45E-11	&	3.73E-15	\\
			&	ARRF	&	1.02E-07	&	1.25E-09	&	2.63E-15	\\
			&	randQB-EI	&	2.55E-07	&	1.20E-08	&	2.59E-15	\\
			&	eSVD	&	3.14E-15	&	4.84E-15	&	3.55E-15	\\
													
		\multirow{5}*{3000} 	&	GRSVD	&	2.82E-07	&	1.87E-09	&	9.28E-15	\\
			&	RSVD	&	1.42E-08	&	3.33E-12	&	4.34E-15	\\
			&	ARRF	&	1.10E-07	&	7.80E-10	&	2.90E-15	\\
			&	randQB-EI	&	2.14E-08	&	3.50E-11	&	2.94E-15	\\
			&	eSVD	&	3.00E-15	&	4.65E-15	&	3.44E-15	\\
													
		\multirow{5}*{4000} 	&	GRSVD	&	5.06E-07	&	1.94E-09	&	4.44E-15	\\
			&	RSVD	&	1.03E-08	&	2.73E-14	&	4.48E-15	\\
			&	ARRF	&	1.02E-07	&	5.67E-09	&	3.24E-15	\\
			&	randQB-EI	&	4.17E-06	&	2.08E-09	&	3.31E-15	\\
			&	eSVD	&	3.42E-15	&	3.45E-14	&	3.51E-15	\\
    \bottomrule[1.2pt]
    \end{tabular}
\end{table}

\begin{figure}[!ht]
    \centering
        \includegraphics[width=0.47\textwidth]{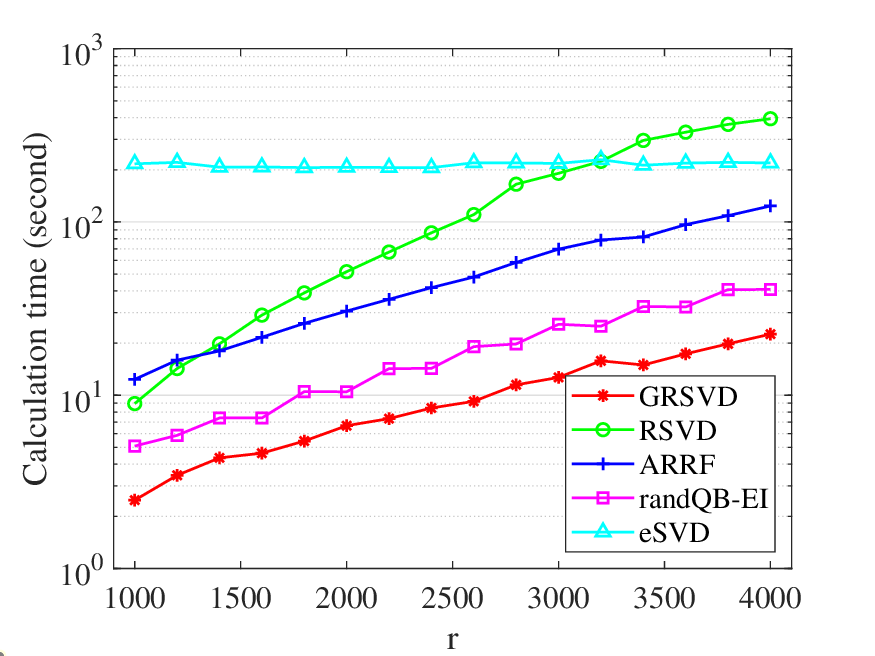}
    \caption{Calculation time of SVD for nearly low-rank matrices.}
	\label{Fig:nlrt}
\end{figure}

In our experiments, all methods (GRSVD, RSVD, ARRF, randQB-EI, and eSVD) yield the same effective rank $r_\epsilon=r$.
It can be seen from Table \ref{Tab.example-nlr} and Figure \ref{Fig:nlrt}, that all methods produce $\hat{U}$ and $\hat{V}$ with good orthogonality. GRSVD maintains stable performance with relatively consistent accuracy. As the matrix rank increases from 1000 to 4000, GRSVD significantly reduces computation time compared to RSVD, ARRF, randQB-EI, and eSVD. In particular, when the rank reaches 4000, its running speed is more than 10 times faster than that of RSVD and eSVD, and about twice as fast as randQB-EI, demonstrating a remarkable computational advantage.

\begin{example}
we run GRI on a series of simulated matrices $A$, where $A$ is generated by the same setup as Example 1 with $r$ ranging from 200 to 2600. The types of matrix we considered are the same as Example 1.
\end{example}

We evaluate the built-in function matrix inverse (BII) by MATLAB command \texttt{inv} with GRI, ARRF in \cite{9}, and randQB-EI in \cite{yu2018} in terms of computational accuracy and efficiency. For ARRF and randQB-EI, the inverse matrix is obtained indirectly: the matrix is first decomposed into its SVD, the inverse of the singular value matrix is then computed, and finally the inverse matrix is reconstructed through matrix multiplication.
We measure their accuracy by the relative errors of matrix inverses:
\[
E_{U} = \frac{\|(I_{m}+AA^{\mathrm{H}})^{-1}-(I_{m}+\tilde{A}\tilde{A}^{\mathrm{H}})^{-1}\|_{\mathrm{F}}}{\|A\|_{\mathrm{F}}},
\]

\[E_{V} = \frac{\|(I_{n}+A^{\mathrm{H}}A)^{-1}-(I_{n}+\tilde{A}^{\mathrm{H}}\tilde{A})^{-1}\|_{\mathrm{F}}}{\|A\|_{\mathrm{F}}},
\]
where $(I_{m}+AA^{\mathrm{H}})^{-1}$ is obtained by BII and $(I_{m}+\tilde{A}\tilde{A}^{\mathrm{H}})^{-1}$ is obtained by GRI, ARRF and randQB-EI.
The experiments are repeated over 100 times and the mean $E_{U}$ and $E_{V}$ of those methods are reported in Table \ref{Tab.example-ninv} with algorithm run time shown in Figure \ref{ninvt}.

\begin{table}[!ht]
  \centering\scriptsize
  \renewcommand{\arraystretch}{1.2} 
  \caption{Error comparison of matrix inverse}
  \begin{tabular}{cccccccc}
    \toprule[1.2pt]
   \multirow{2}*{$r$} & \multicolumn{3}{c}{Error $E_{U}$}  &  \multicolumn{3}{c}{Error $E_{V}$}  \\
&GRI&ARRF &randQB-EI&GRI&ARRF &randQB-EI\\
\midrule
	400	&	9.65E-07	&4.71E-08&6.67E-08	&	9.29E-10	&9.28E-15&7.13E-15	\\
	800	&	5.53E-07	&1.03E-08&1.02E-08	&	3.15E-11	&1.93E-14&6.83E-15	\\
	1200	&	9.44E-07	&4.51E-08&8.57E-08	&	2.39E-10	&1.00E-14&3.03E-15	\\
	1600	&	7.07E-07	&1.74E-08&5.03E-08	&	7.05E-12	&7.11E-15&6.12E-14	\\
	2000	&	3.47E-07	&4.18E-08&1.22E-08	&	1.45E-11	&1.06E-14&2.22E-15	\\
	2400	&	2.99E-07	&3.78E-08&3.78E-08	&	1.22E-10	&1.05E-14&1.03E-14	\\
    \bottomrule[1.2pt]
    \end{tabular}
\label{Tab.example-ninv}
\end{table}

\begin{figure}[!ht]
    \centering
    \subfloat[$(I_m+AA^{\mathrm{H}})^{-1}$]{
        \includegraphics[width=2.4in]{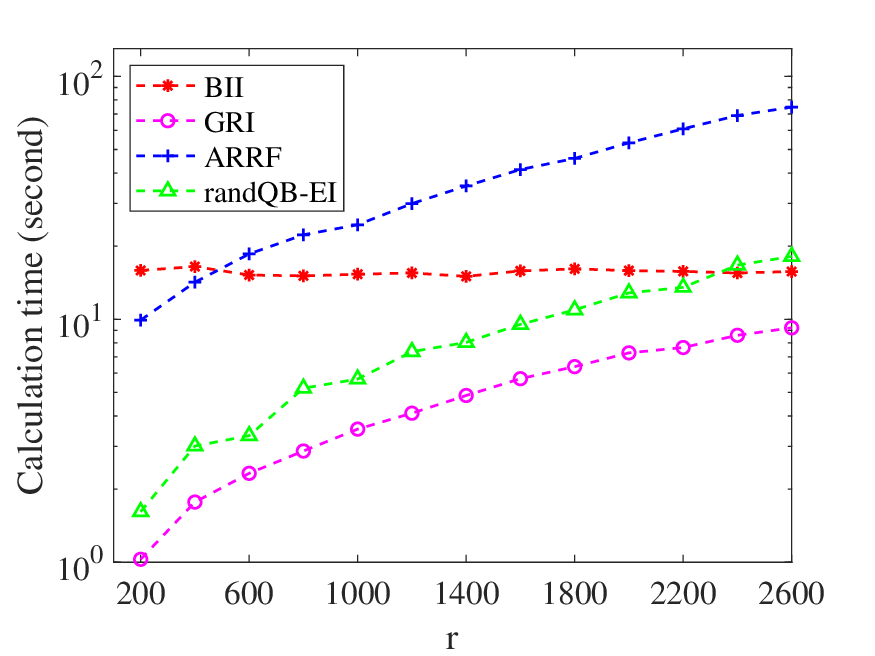}
    }
    \subfloat[$(I_n+A^{\mathrm{H}}A)^{-1}$]{
	\includegraphics[width=2.4in]{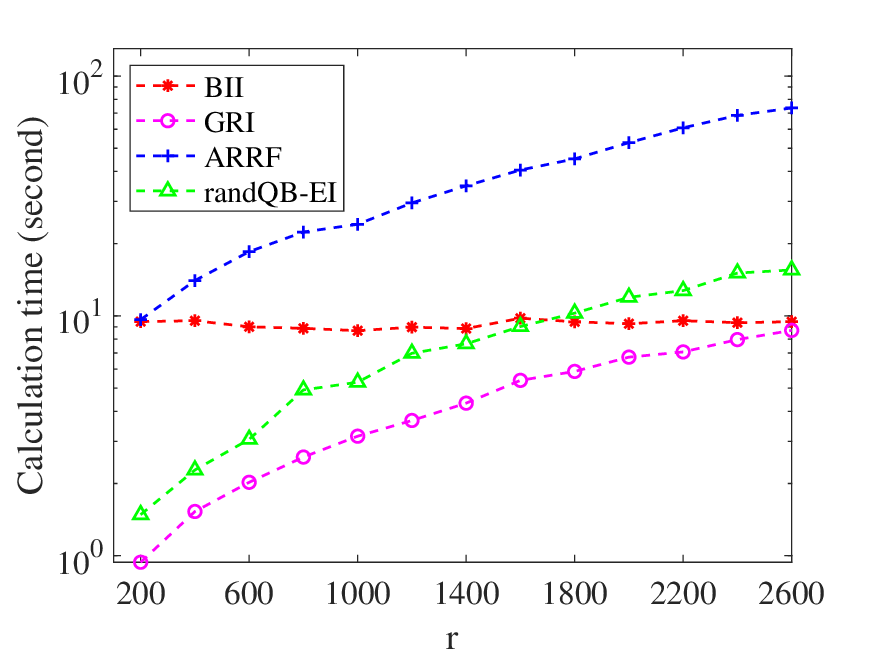}
    }
    \caption{Calculation time of matrix inverse.}
    \label{ninvt}
\end{figure}


For nearly low-rank matrices, it is seen from Table \ref{Tab.example-ninv} that GRI demonstrates acceptable accuracy in computing the inverses of the two types of matrices. While its accuracy is not the highest among all methods, it remains sufficient for practical purposes. In terms of calculation time, GRI takes less calculation time than BII, ARRF, and randQB-EI  for all case of $r$, as also observed from Figure \ref{ninvt}. In contrast, ARRF requires even more time than BII because it computes the matrix inverse column by column, which is relatively time-consuming. Therefore, GRI offers a fast and sufficiently accurate approach for inverting nearly low-rank matrices.

  
\begin{example}[Echocardiography denoising, \cite{xu2023,xu2025}]
This example aims to explore the denoising of echocardiography based on SVD, with the core objective of effectively removing noise while preserving the necessary energy proportion and critical image information. The experimental data are derived from the EchoNet-Dynamic database \cite{nature}, which comprises $10,030$ apical four-chamber echocardiography videos \footnote{\url{https://echonet.github.io/dynamic/}} obtained from imaging results collected during routine clinical care at Stanford Hospital between 2016 and 2018. Each video is cropped and masked to eliminate text and other redundant information outside the scanned sector, followed by downsampling to a standardized resolution of $112 \times 112$ pixels using tricubic interpolation. 
\end{example}

The echocardiography data, with dimensions $m \times n \times T$, are reshaped into a concatenated matrix $A$ of size $(m \times n) \times T$. Based on matrix $A$, five SVD methods are employed for denoising, as illustrated in Figure \ref{lct}:

\begin{figure}[!ht]
  \centering
  \includegraphics[width=1\textwidth]{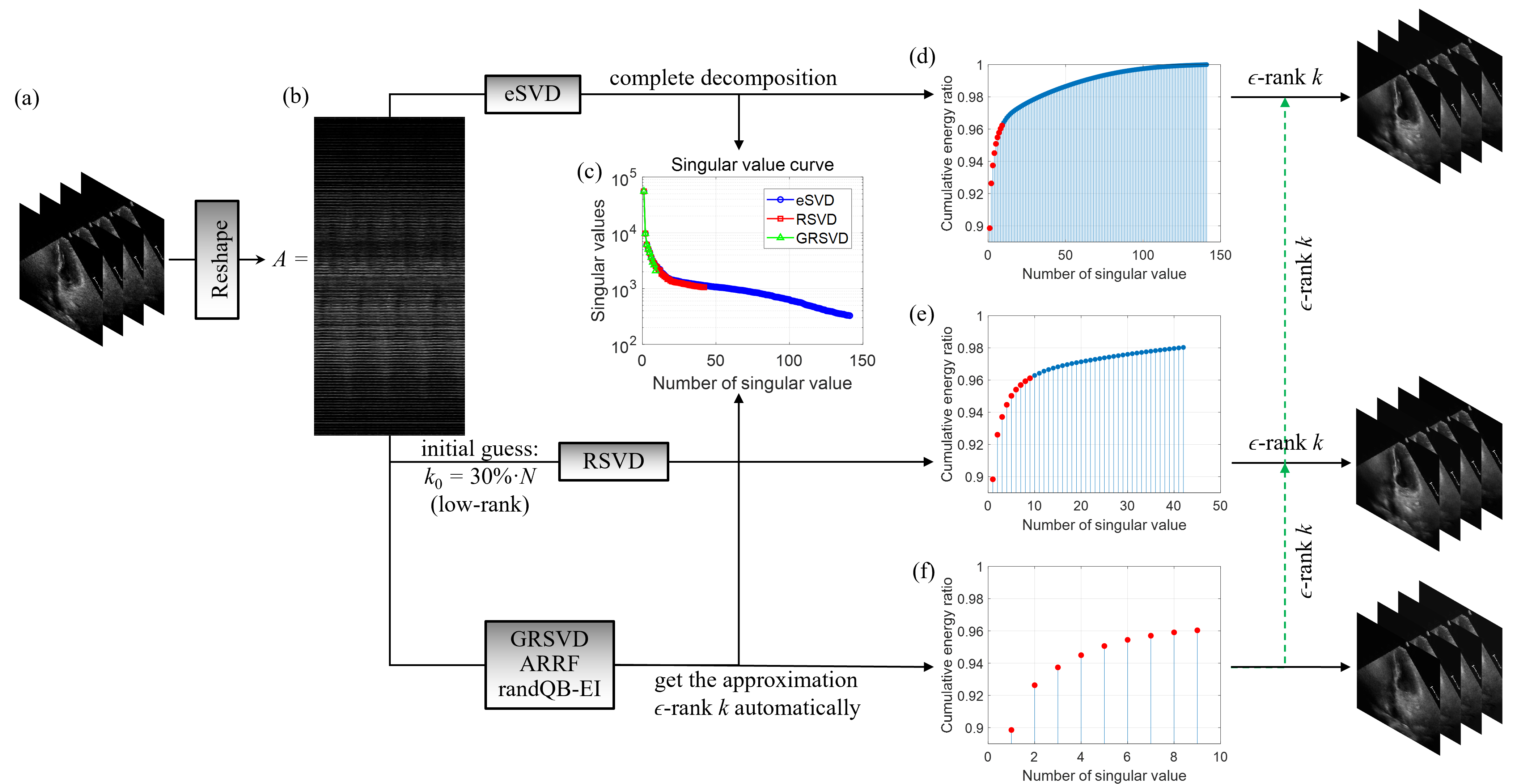}
  \caption{Schematic of $\epsilon$-rank denoising using eSVD, randomized SVD, GRSVD, Algorithm 4.2 in \cite{9} and randQB-EI. (a) A 3-D stack of echocardiography. (b) A spatiotemporal representation (Casorati matrix) where all pixels at one time point are arranged in one column. (c) Singular value curves of eSVD, RSVD and GRSVD. (d) Stem image of the cumulative energy ratio of eSVD. (e) Stem image of the cumulative energy ratio of RSVD. (f) Stem image of the cumulative energy ratio of GRSVD.}\label{lct}
\end{figure}

\begin{itemize}
\item The concatenated matrix $A$ is completely decomposed using eSVD. Calculate the cumulative energy ratios of all singular values. Then determine the $\epsilon$-rank $k$ according to Definition \ref{Ass1}. After that, obtain the denoised results through matrix multiplication.
\item Use RSVD to partially decompose the concatenated matrix $A$.
    The $\epsilon$-rank $k_0$ is initially guessed by an empirical selection for low-rank (e.g., first 30\% of the rank) \cite{xu2023,xu2025}, and then the ideal threshold can be determined as shown in Figure \ref{lct}(e).
    Then, determine the $\epsilon$-rank $k$ according to Definition \ref{Ass1}, where the denominator in Definition \ref{Ass1} can be obtained by the square of the Frobenius norm of the concatenated matrix $A$. Finally, obtain the denoised video data through matrix multiplication.
\item Using GRSVD, ARRF, and randQB-EI to perform a partial singular value decomposition of the concatenated matrix $A$. Since different algorithms produce different truncation parameters $k$, all three methods use the $\epsilon$-rank $k$ calculated by GRSVD as the number of singular values to be retained, and obtain the denoising result by matrix multiplication. To ensure consistency in comparisons, the $\epsilon$-rank for eSVD and RSVD are also based on the $\epsilon$-rank $k$ obtained from GRSVD. Such a setup guarantees that all methods are evaluated under the same truncation parameter, thereby ensuring a fair and meaningful comparison of the subsequent results.
\end{itemize}

\begin{table}
  \centering\scriptsize
  \setlength{\tabcolsep}{4pt} 
  \renewcommand{\arraystretch}{1.2} 
  \caption{Results of SVD related indicators of echocardiography}\label{Table:ex}
  \begin{tabular}{ccccccccc}
    \toprule[1.2pt]
    Dimension & $\epsilon$	&	Methods	&MSE	&	PSNR(dB)	&	EPI	&	Time(s)	&	$k$	&	ER(\%)	\\
    \midrule
\multirow{5}{*}{$112\times112\times141$}	&	\multirow{5}{*}{0.05}	&	GRSVD	&	0.1985	&	55.1656	&	0.9386	&	\textbf{	0.0049	}	&	\multirow{5}{*}{	9	}	&	\multirow{5}{*}{95.84}	\\
	&		&	RSVD	&	0.1965	&	55.1970	&	0.9419	&		0.0273		&				&		\\
	&		&	eSVD	&	0.1937	&	55.2600	&	0.9427	&		0.0295		&				&		\\
	&		&	ARRF	&	0.1954	&	55.2221	&	0.9420	&		0.1037		&				&		\\
	&		&	randQB-EI	&	0.2283	&	54.5508	&	0.9261	&		0.0104		&				&		\\
\midrule																					
\multirow{5}{*}{$112\times112\times234$}	&	\multirow{5}{*}{0.03}	&	GRSVD	&	0.1654	&	55.9551	&	0.9447	&	\textbf{	0.0131	}	&	\multirow{5}{*}{	13	}	&	\multirow{5}{*}{97.43}	\\
	&		&	RSVD	&	0.1603	&	56.0819	&	0.9483	&		0.0513		&				&		\\
	&		&	eSVD	&	0.1592	&	56.1114	&	0.9487	&		0.0670		&				&		\\
	&		&	ARRF	&	0.1616	&	56.0475	&	0.9478	&		0.1081		&				&		\\
	&		&	randQB-EI	&	0.1781	&	55.6288	&	0.9382	&		0.0263		&				&		\\
\midrule																					
\multirow{5}{*}{$112\times112\times154$}	&	\multirow{5}{*}{0.03}	&	GRSVD	&	0.1518	&	56.3355	&	0.9533	&	\textbf{	0.0090	}	&	\multirow{5}{*}{	18	}	&	\multirow{5}{*}{97.53}	\\
	&		&	RSVD	&	0.1488	&	56.4073	&	0.9552	&		0.0330		&				&		\\
	&		&	eSVD	&	0.1458	&	56.4945	&	0.9560	&		0.0334		&				&		\\
	&		&	ARRF	&	0.1482	&	56.4231	&	0.9552	&		0.1012		&				&		\\
	&		&	randQB-EI	&	0.1713	&	55.8036	&	0.9438	&		0.0177		&				&		\\
\midrule																					
\multirow{5}{*}{$112\times112\times636$}	&	\multirow{5}{*}{0.09}	&	GRSVD	&	0.2919	&	53.5075	&	0.8852	&	\textbf{	0.0374	}	&	\multirow{5}{*}{	17	}	&	\multirow{5}{*}{92.88}	\\
	&		&	RSVD	&	0.2730	&	53.7696	&	0.8969	&		0.2756		&				&		\\
	&		&	eSVD	&	0.2718	&	53.7879	&	0.8974	&		0.4195		&				&		\\
	&		&	ARRF	&	0.2948	&	53.4370	&	0.8850	&		0.1320		&				&		\\
	&		&	randQB-EI	&	0.3517	&	52.6735	&	0.8565	&		0.0696		&				&		\\
\midrule																					
\multirow{5}{*}{$112\times112\times864$}	&	\multirow{5}{*}{0.05}	&	GRSVD	&	0.2127	&	54.8640	&	0.9311	&	\textbf{	0.0321	}	&	\multirow{5}{*}{	12	}	&	\multirow{5}{*}{96.31}	\\
	&		&	RSVD	&	0.2094	&	54.9201	&	0.9296	&		0.4702		&				&		\\
	&		&	eSVD	&	0.2091	&	54.9283	&	0.9296	&		0.6294		&				&		\\
	&		&	ARRF	&	0.2151	&	54.8056	&	0.9272	&		0.1468		&				&		\\
	&		&	randQB-EI	&	0.2429	&	54.2856	&	0.9150	&		0.0692		&				&		\\
\midrule																					
\multirow{5}{*}{$112\times112\times614$}	&	\multirow{5}{*}{0.05}	&	GRSVD	&	0.2489	&	54.2778	&	0.9186	&	\textbf{	0.0363	}	&	\multirow{5}{*}{	16	}	&	\multirow{5}{*}{95.74}	\\
	&		&	RSVD	&	0.2283	&	54.5458	&	0.9348	&		0.2818		&				&		\\
	&		&	eSVD	&	0.2275	&	54.5602	&	0.9351	&		0.4087		&				&		\\
	&		&	ARRF	&	0.2373	&	54.3789	&	0.9309	&		0.1369		&				&		\\
	&		&	randQB-EI	&	0.2735	&	53.7689	&	0.9126	&		0.0706		&				&		\\
    \bottomrule[1.2pt]
  \end{tabular}

\end{table}

To validate the effectiveness of the above methods, six specific echocardiography videos are selected for experimentation with varied $\epsilon$.
Let $A$ denote the original concatenated matrix. We obtain the $\epsilon$-rank $k$, computation time (in seconds), energy ratio (ER) defined in Definition \ref{Ass1}, and the three standard criteria for evaluating the $\epsilon$-rank approximation matrix $\hat{A}$ for all methods, as follows:
\begin{itemize}
\item[(i)] The mean square error (MSE) \[\text{MSE} = \frac{\|A-\hat{A}\|_{\mathrm{F}}}{\|A\|_{\mathrm{F}}}.\]
\item[(ii)] The peak signal-to-noise ratio value (PSNR) \[\text{PSNR} = 10 \cdot \log_{10}\left(\frac{255^2}{\text{MSE}}\right).\]
\item[(iii)] The edge preservation index (EPI) 
\[\text{EPI} = \frac{\sum_{i,j} (\Delta A - \overline{\Delta A}) \cdot (\widehat{\Delta A} - \overline{\widehat{\Delta A}}) }{\sqrt{\sum_{i,j} (\Delta A - \overline{\Delta A})^2} \cdot \sqrt{\sum_{i,j} (\widehat{\Delta A} - \overline{\widehat{\Delta A}})^2}},\]
where $\bar{A}(i,j)$ and $\bar{\hat{A}}(i,j)$ are mean values in the region of interest of $A(i,j)$ and $\hat{A}(i,j)$, respectively.
$\Delta A(i, j)$ is a highpass filtered version of $A(i, j)$, obtained with a $3\times 3$-pixel standard approximation of the Laplacian operator.
\end{itemize}

The experimental results are detailed in Table \ref{Table:ex} and Figure \ref{Fig:ex}. We have the following conclusions:

\begin{figure}[!ht]
  \centering
  \setlength{\tabcolsep}{2pt} 
  \begin{tabular}{llllll}
     (a) & (b) & (c) & (d) & (e) & (f)\\
    \includegraphics[width=0.15\textwidth]{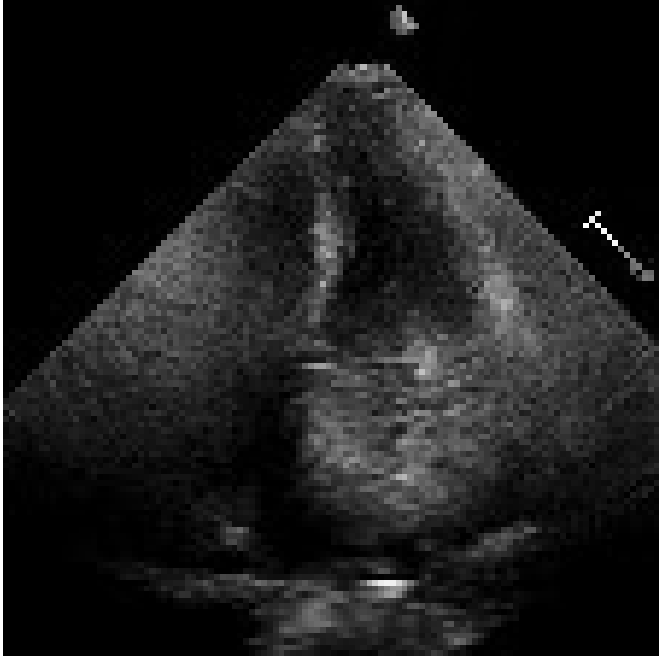} &
    \includegraphics[width=0.15\textwidth]{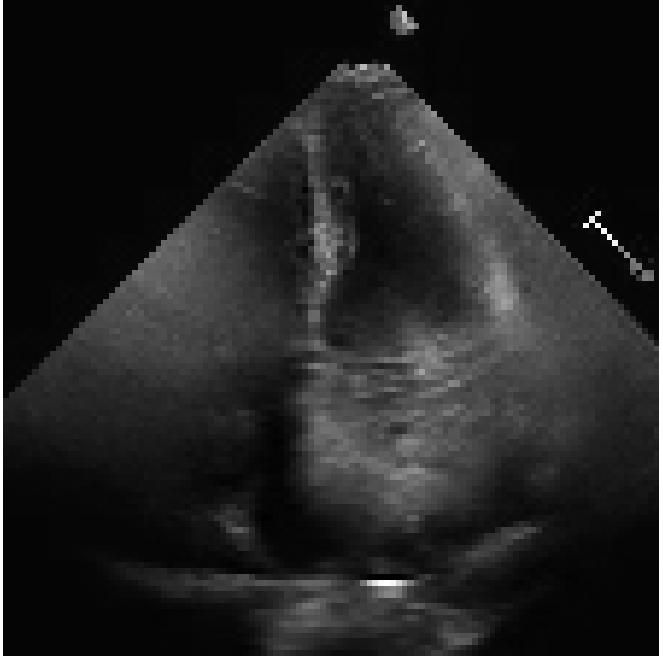} &
    \includegraphics[width=0.15\textwidth]{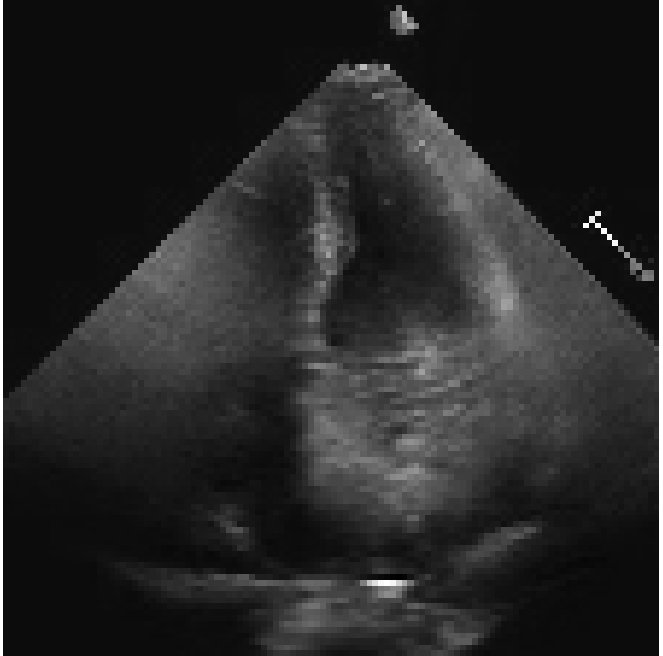} &
    \includegraphics[width=0.15\textwidth]{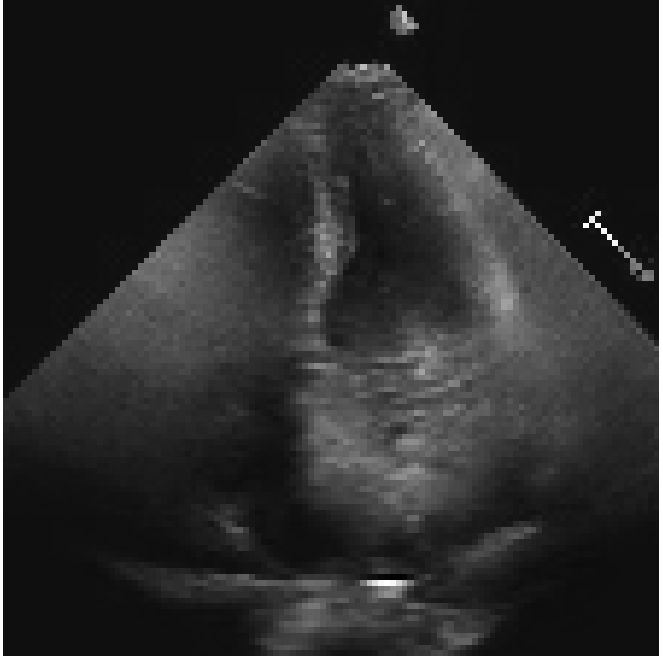} &
    \includegraphics[width=0.15\textwidth]{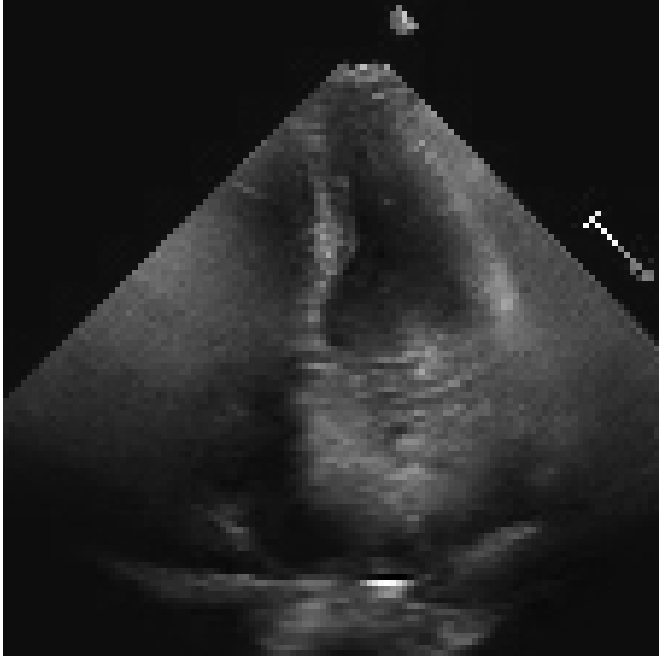} &
    \includegraphics[width=0.15\textwidth]{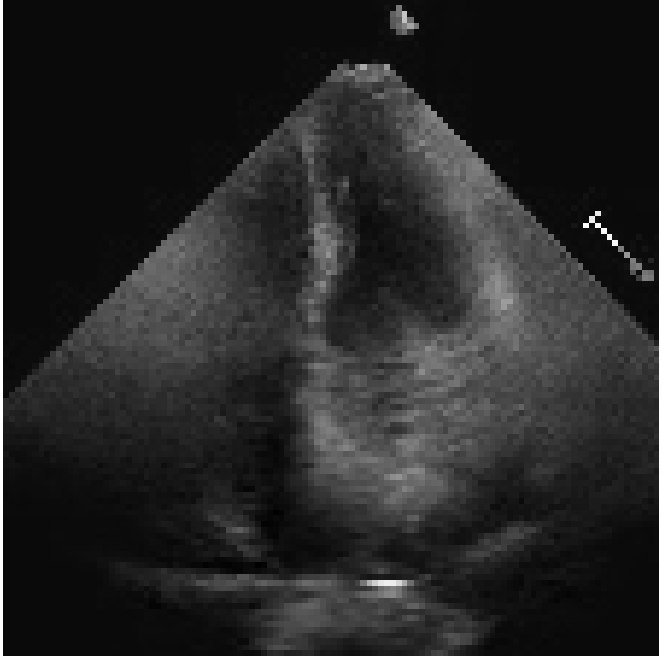}\\
      \includegraphics[width=0.15\textwidth]{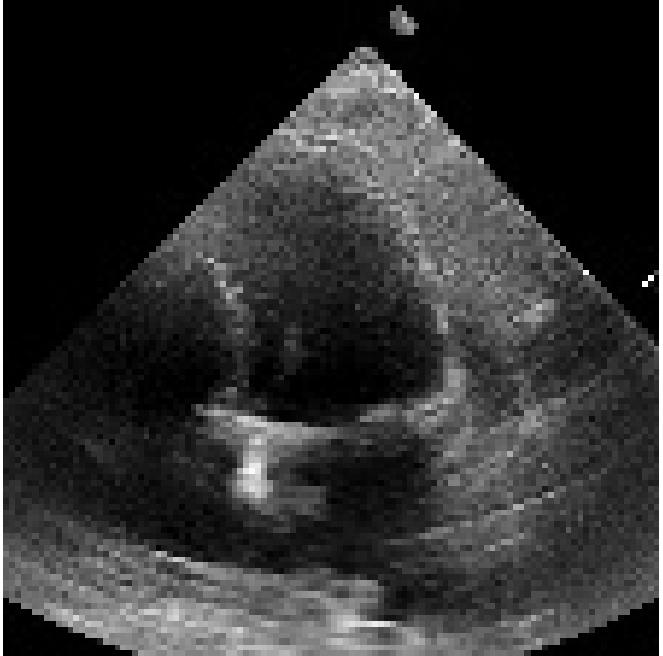} &
      \includegraphics[width=0.15\textwidth]{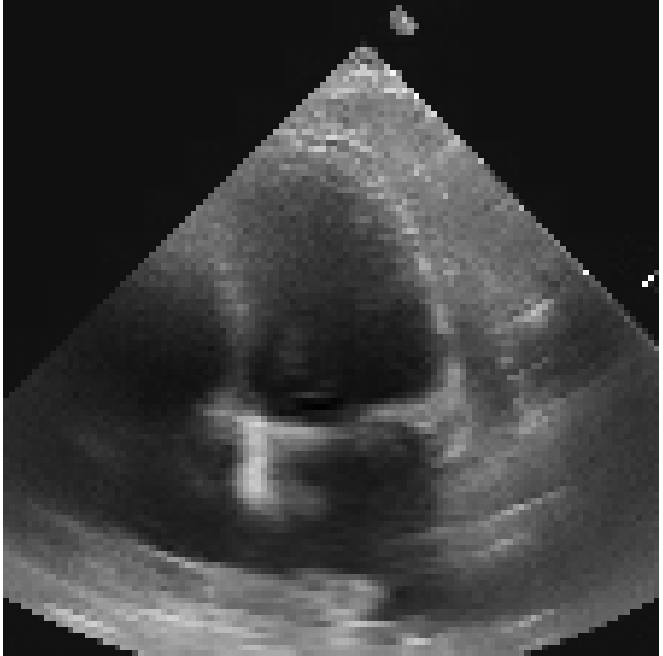} &
      \includegraphics[width=0.15\textwidth]{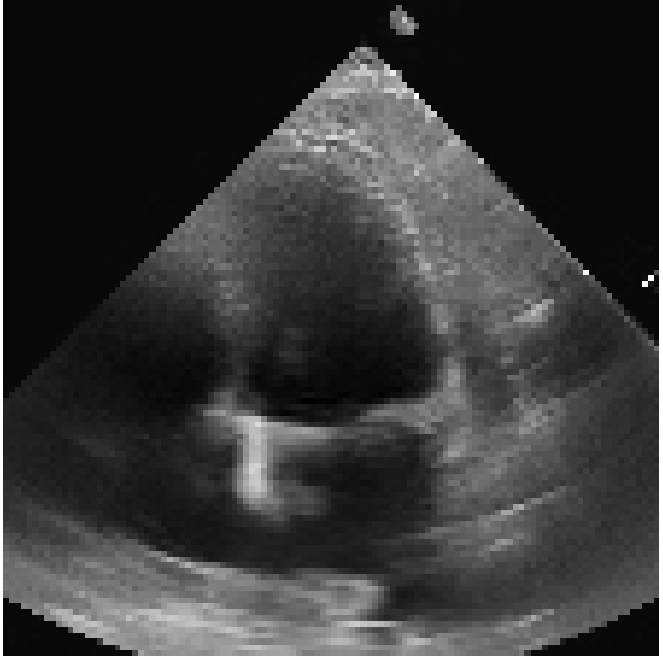} &
      \includegraphics[width=0.15\textwidth]{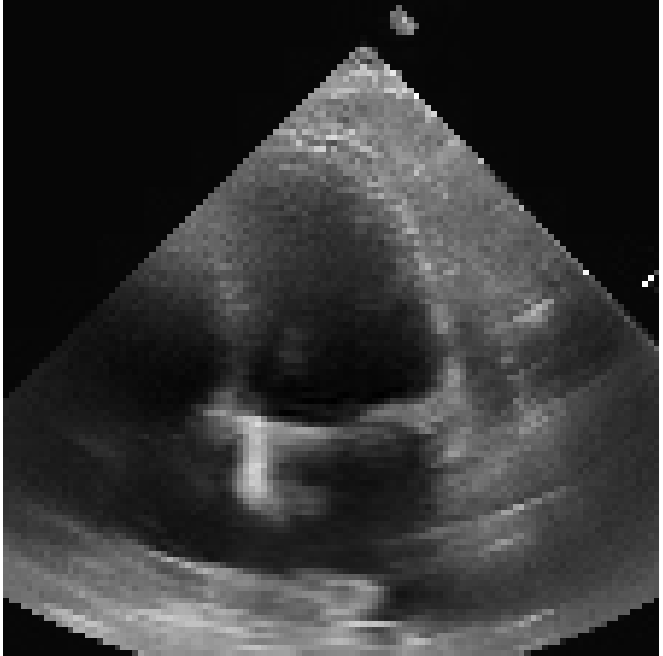} &
      \includegraphics[width=0.15\textwidth]{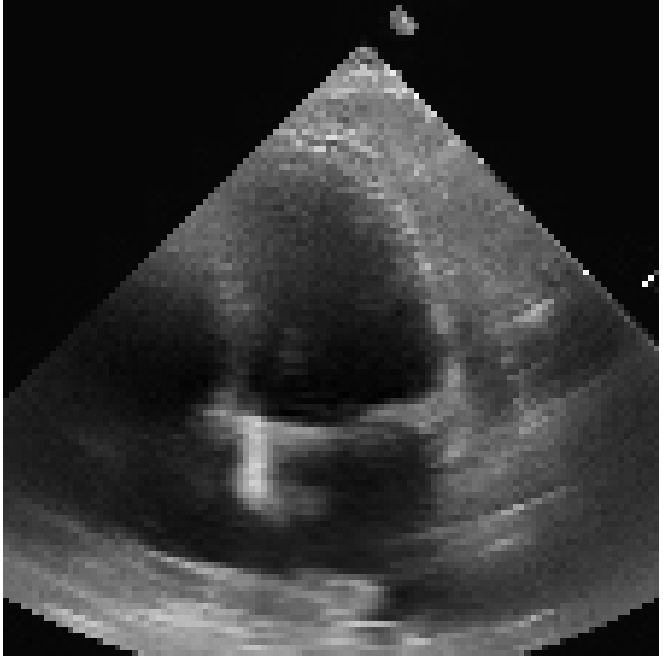} &
      \includegraphics[width=0.15\textwidth]{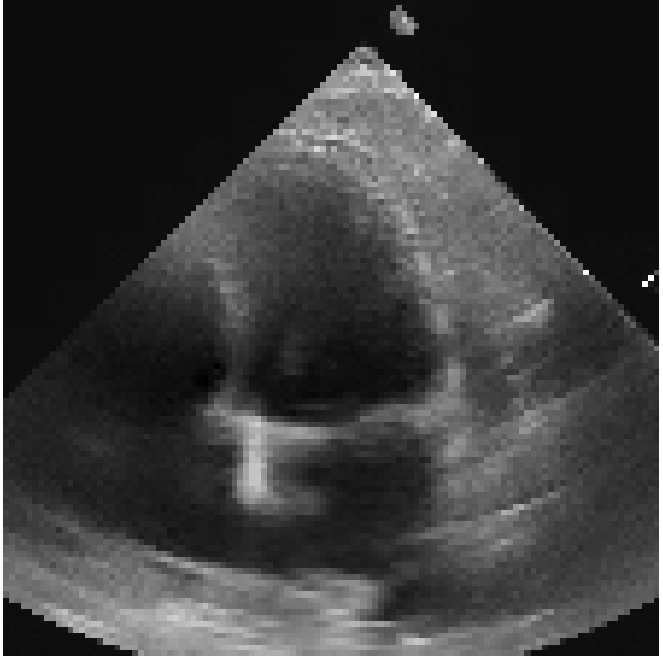}\\
    \includegraphics[width=0.15\textwidth]{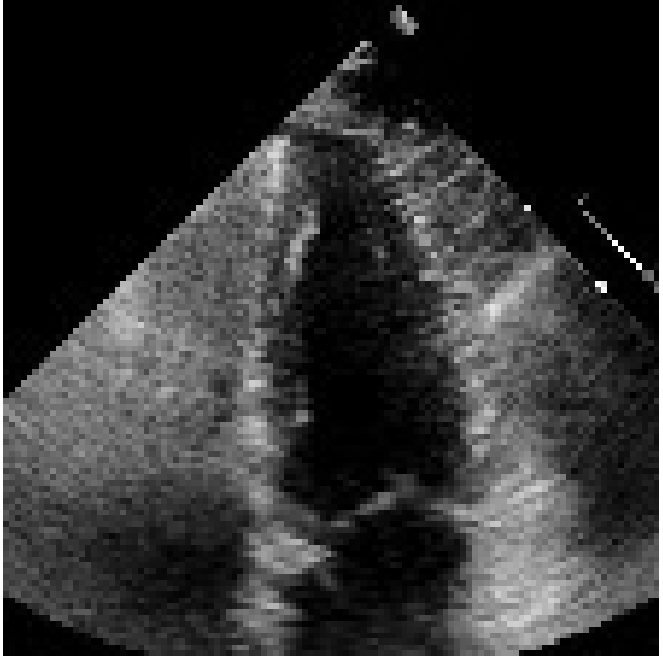} &
    \includegraphics[width=0.15\textwidth]{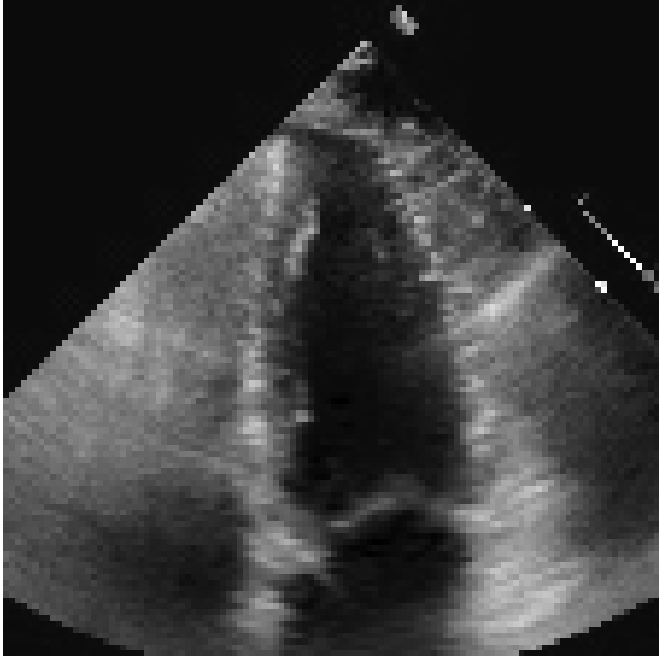} &
    \includegraphics[width=0.15\textwidth]{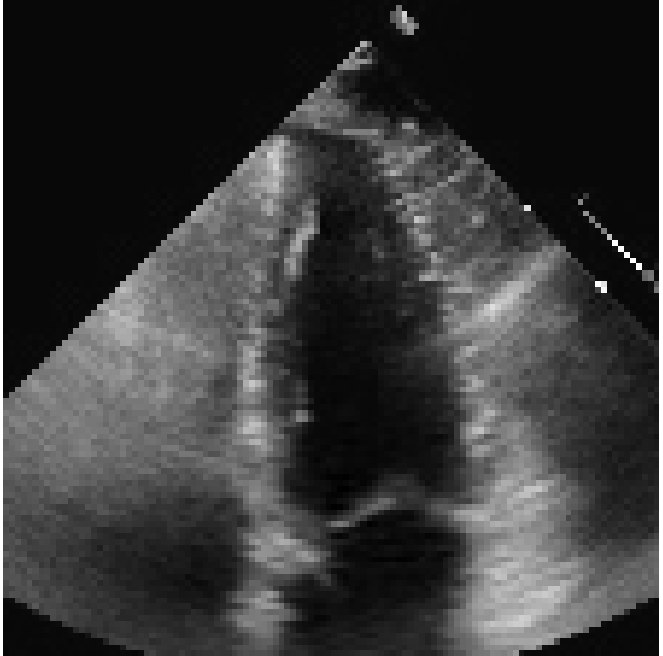} &
    \includegraphics[width=0.15\textwidth]{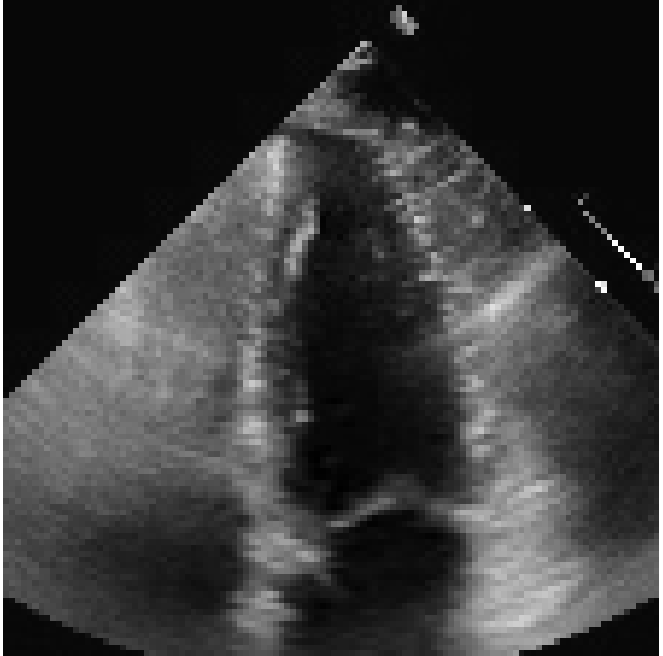} &
    \includegraphics[width=0.15\textwidth]{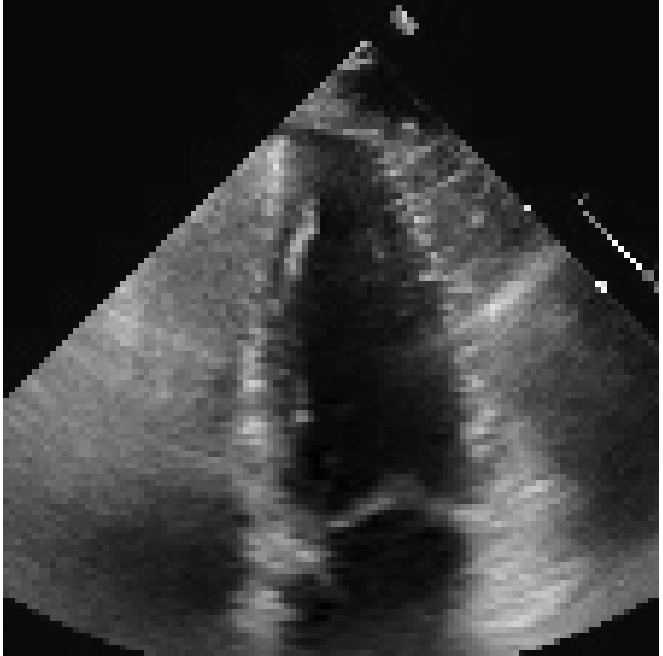} &
    \includegraphics[width=0.15\textwidth]{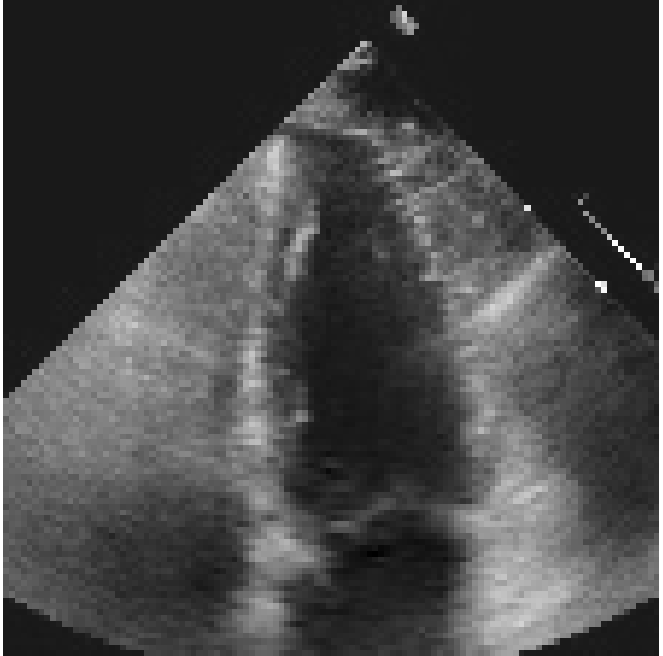}\\
    \includegraphics[width=0.15\textwidth]{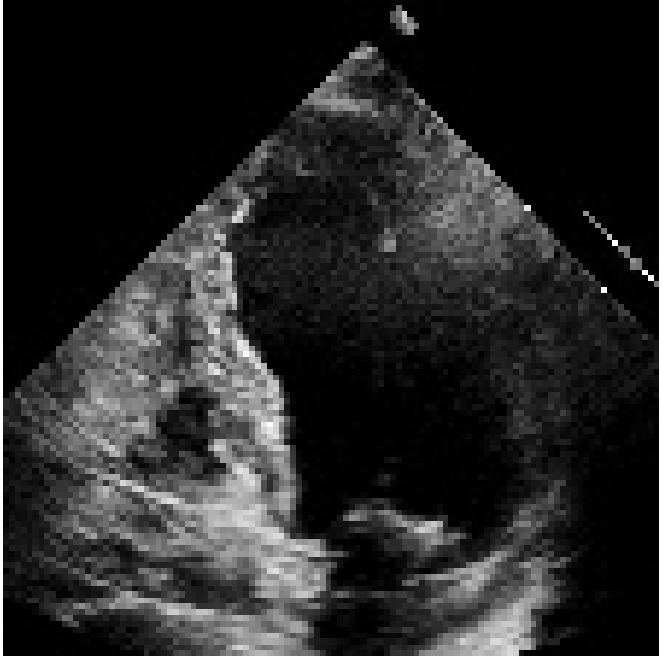} &
    \includegraphics[width=0.15\textwidth]{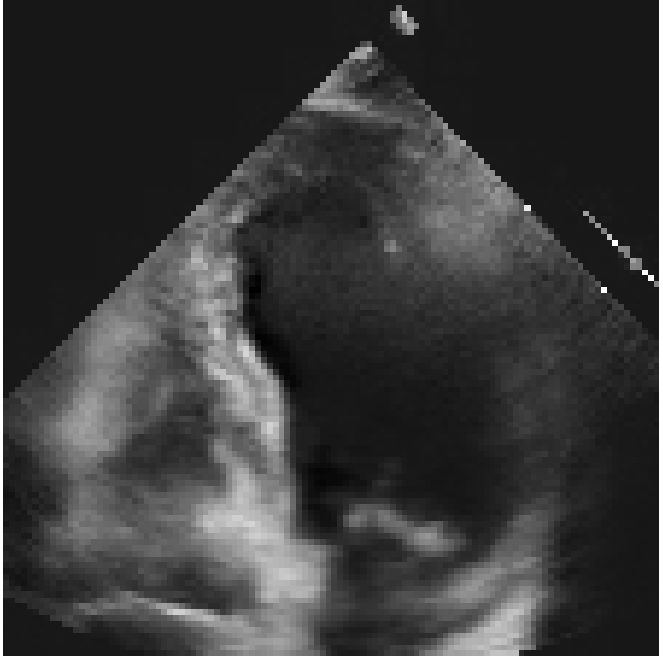} &
    \includegraphics[width=0.15\textwidth]{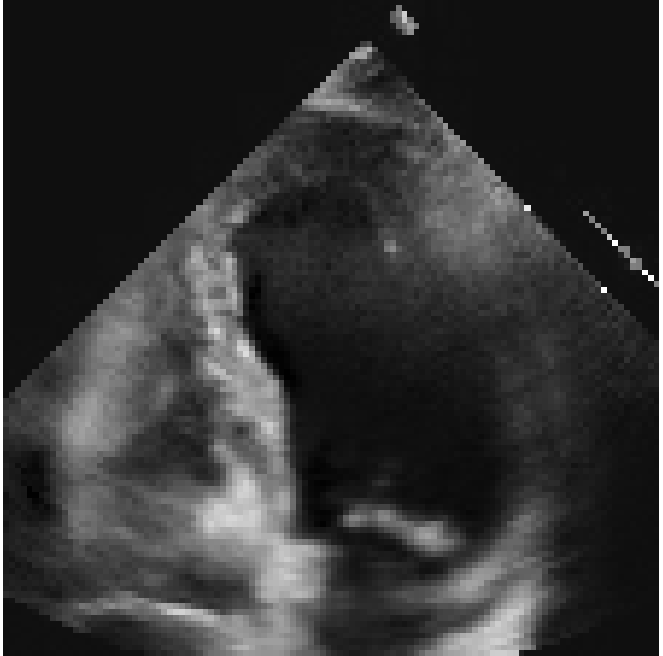} &
    \includegraphics[width=0.15\textwidth]{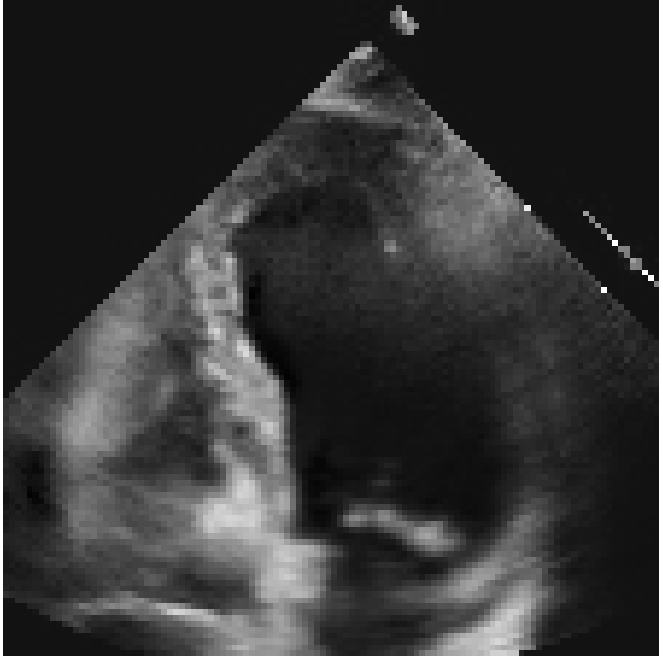} &
    \includegraphics[width=0.15\textwidth]{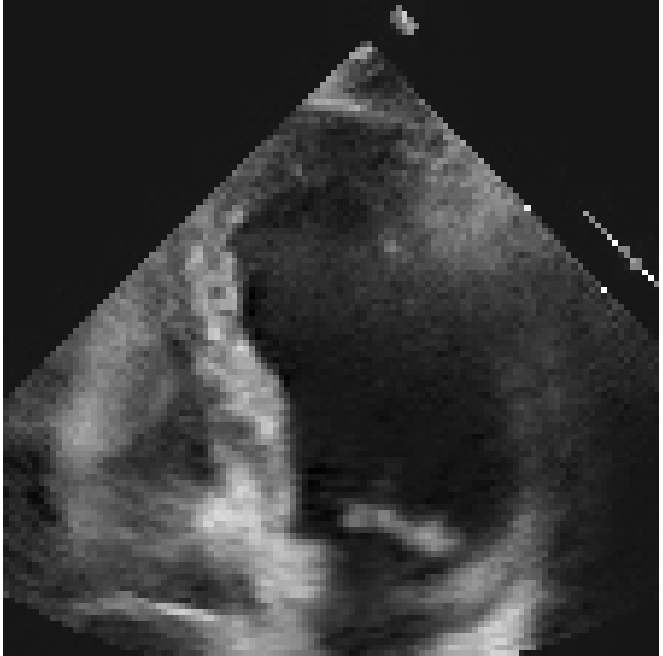} &
    \includegraphics[width=0.15\textwidth]{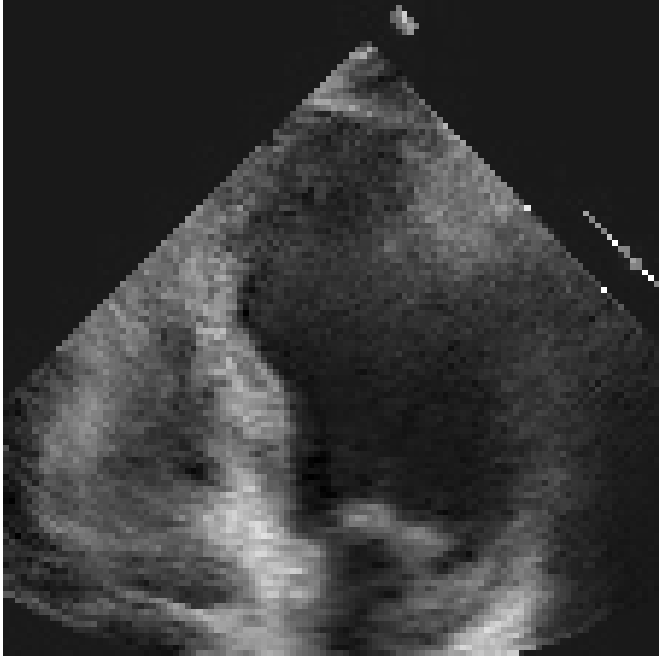}\\
    \includegraphics[width=0.15\textwidth]{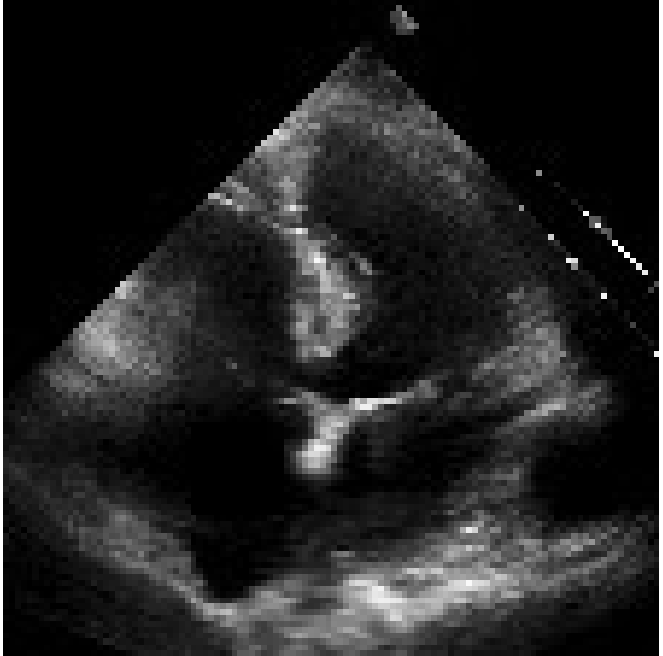} &
    \includegraphics[width=0.15\textwidth]{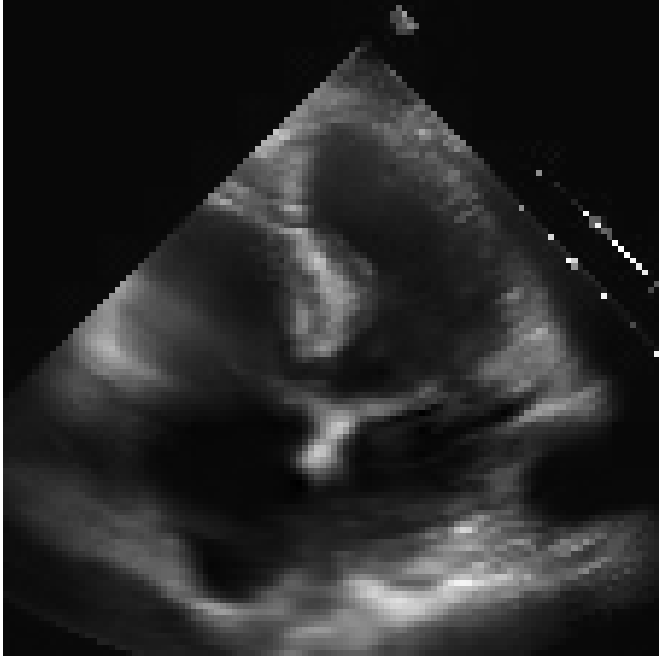} &
    \includegraphics[width=0.15\textwidth]{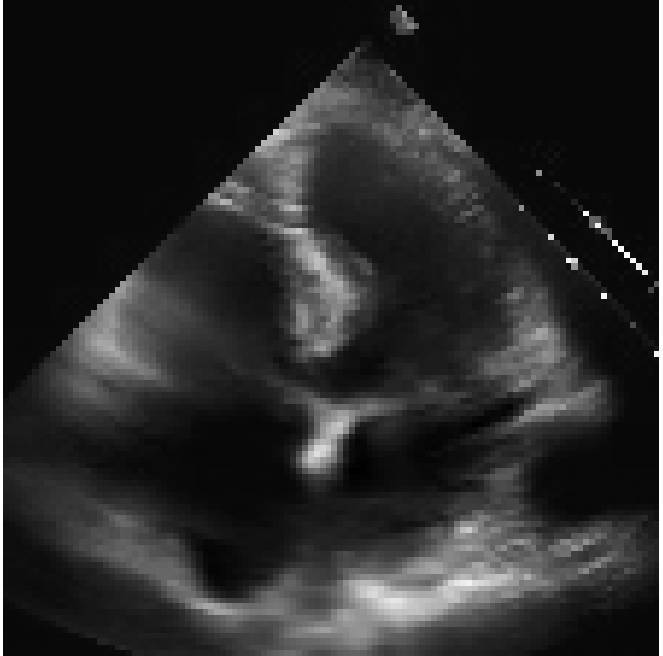} &
    \includegraphics[width=0.15\textwidth]{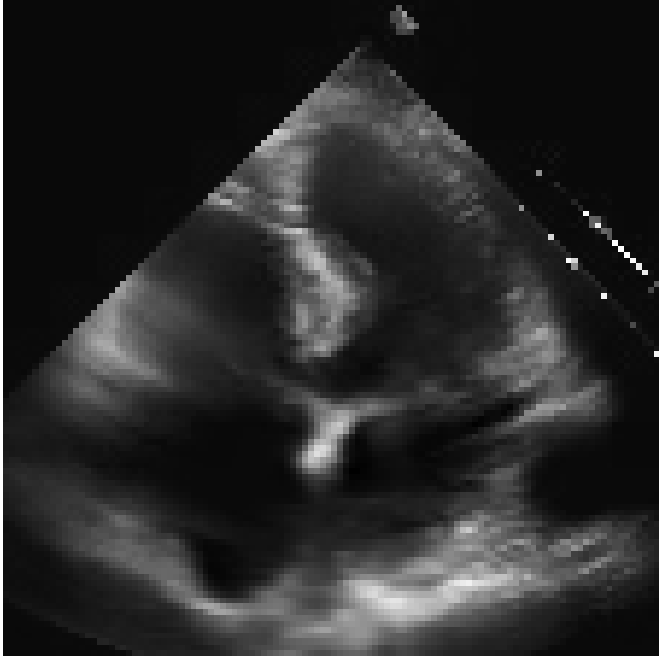} &
    \includegraphics[width=0.15\textwidth]{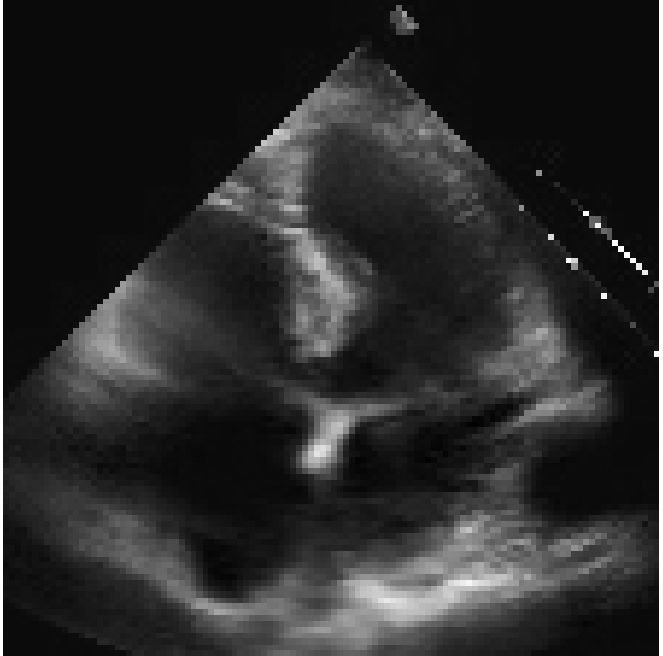} &
    \includegraphics[width=0.15\textwidth]{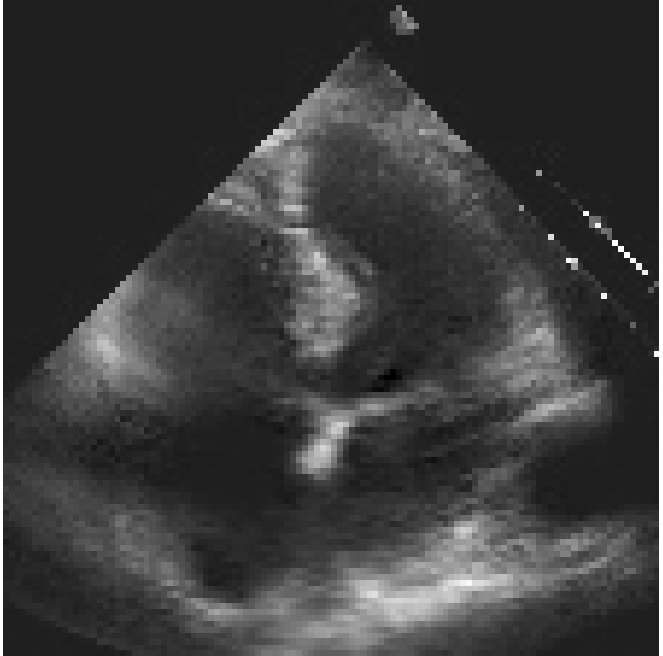}\\
    \includegraphics[width=0.15\textwidth]{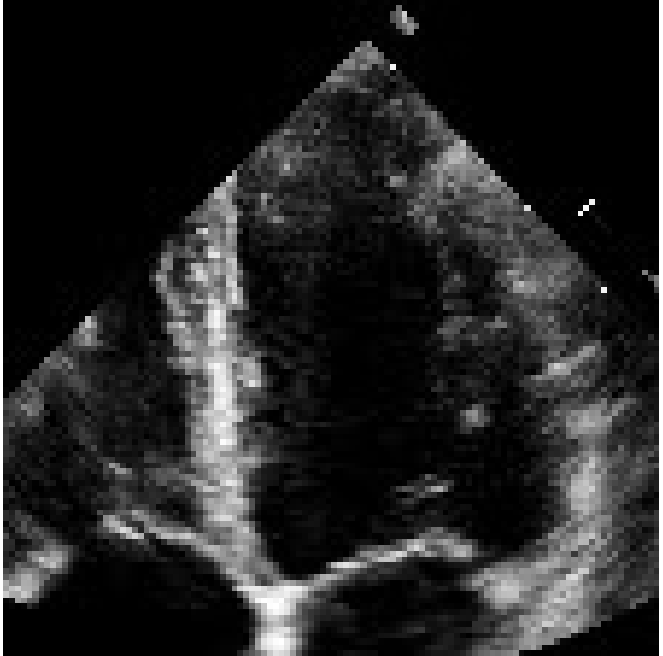} &
    \includegraphics[width=0.15\textwidth]{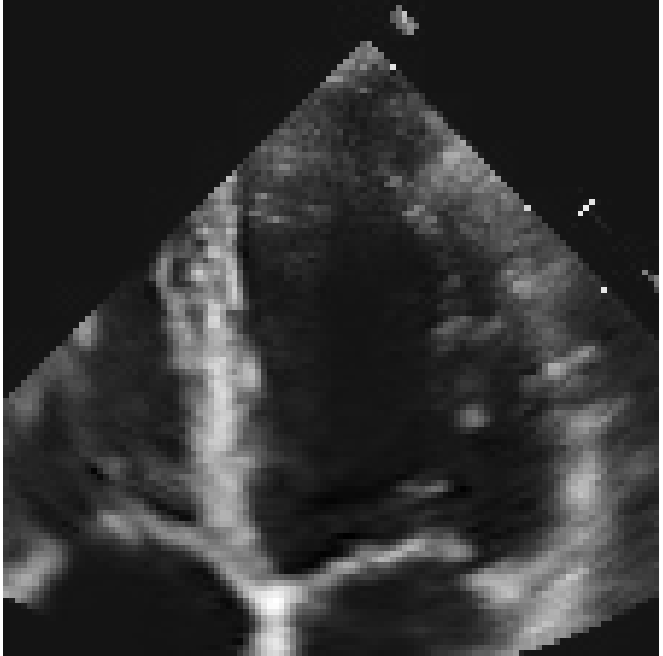} &
    \includegraphics[width=0.15\textwidth]{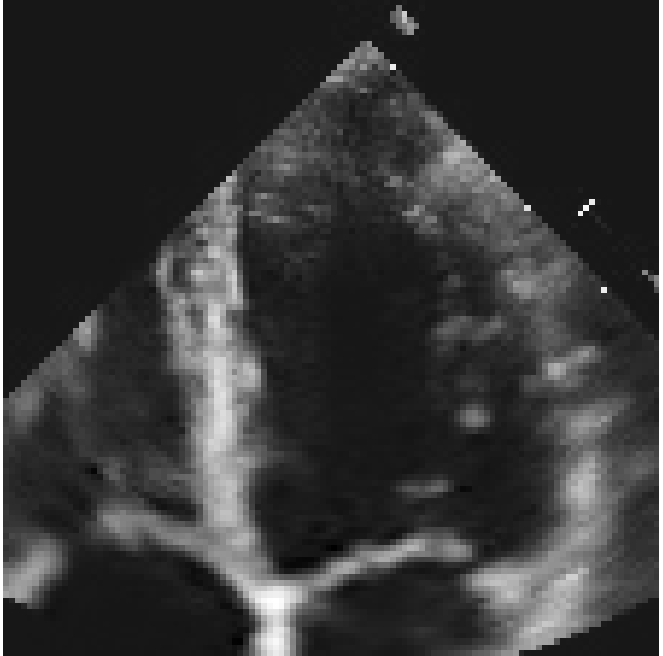} &
    \includegraphics[width=0.15\textwidth]{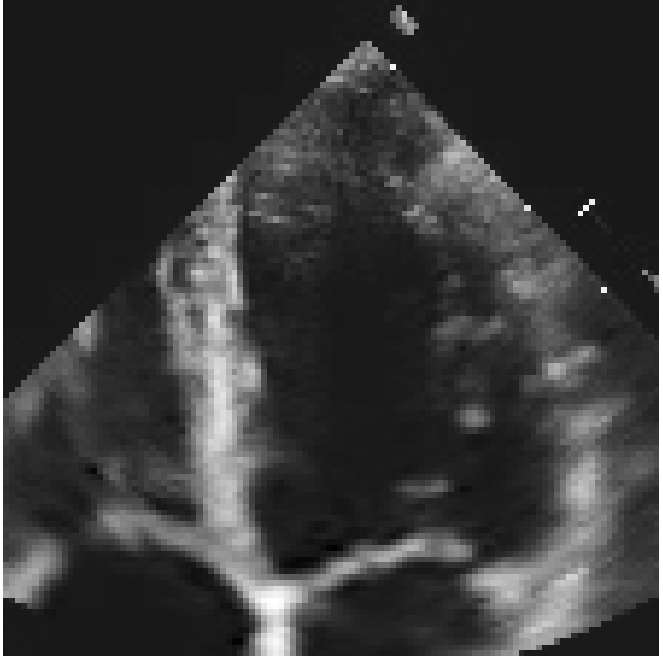} &
    \includegraphics[width=0.15\textwidth]{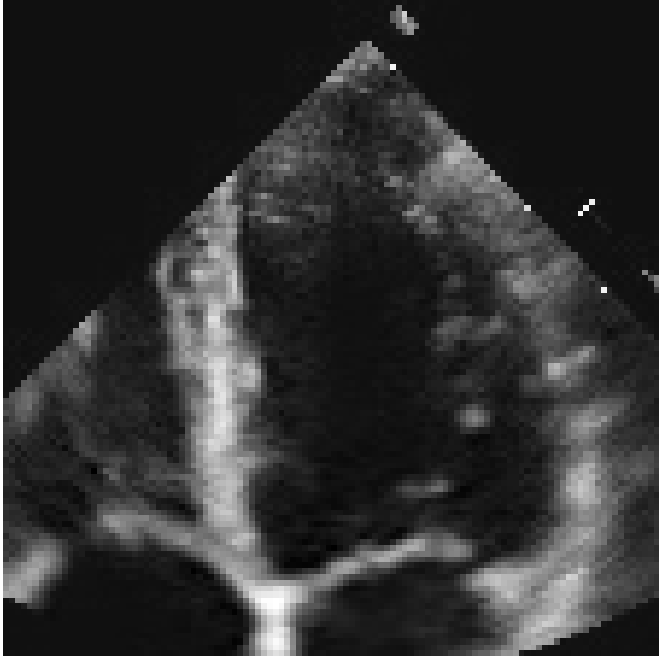} &
    \includegraphics[width=0.15\textwidth]{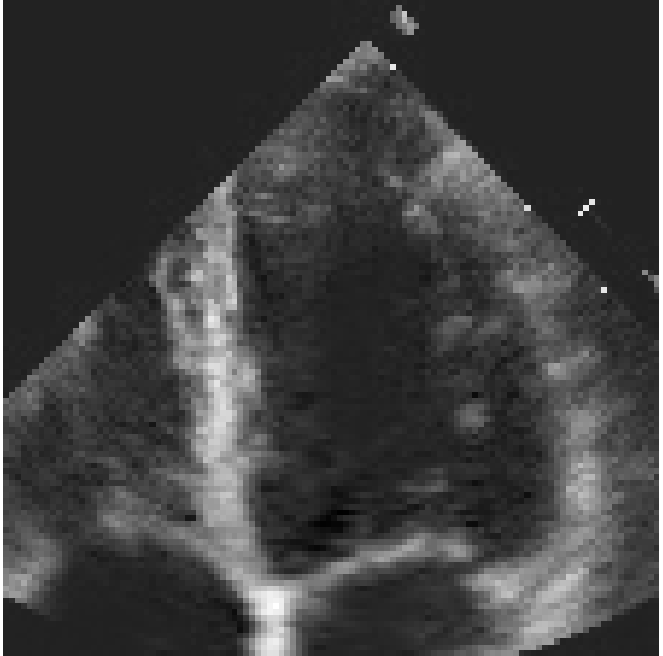}
  \end{tabular}\\
  \caption{Echocardiographic images processed by different methods. (a) Original. (b) GRSVD. (c) eSVD. (d) RSVD. (e) ARRF. (f) randQB-EI.}
    \label{Fig:ex}
\end{figure}

\begin{itemize}
  \item All denoising algorithms maintain extremely high fidelity between the processed echocardiographic and the original echocardiographic. The MSE of all algorithms is low, indicating that the average pixel differences are small, and the PSNR above 50 dB shows that the image clarity is high. In addition, the EPI is close to 1, indicating that critical edge details are well preserved.
  \item Ensuring the required energy proportion with a smaller $\epsilon$-rank. As indicated by the results in the table, under the same $\epsilon$ constraints, GRSVD can retain at least $1-\epsilon$ of the total energy while using a smaller or comparable truncated rank. This smaller $\epsilon$-rank reduces both computational and storage overhead, which is especially beneficial for large-scale data.
  \item Significantly faster computational speed. Compared to eSVD, GRSVD often achieves at least an order-of-magnitude speedup. When compared to RSVD, ARRF, and randQB-EI, GRSVD generally provides multiple-fold acceleration. This speed advantage is crucial for clinical or research applications involving large volumes of ultrasound data.
  \item Comparable image quality to full decomposition, with stable denoising performance. Judging by PSNR, MSE and EPI, GRSVD produces results very close to eSVD, indicating effective noise reduction while preserving critical information in echocardiographic. The method maintains robust performance across different resolutions and $\epsilon$ values, making it well-suited for large-scale echocardiography denoising.
\end{itemize}

\begin{example}
Multicollinearity often undermines the stability and interpretability of regression models. This issue is particularly pronounced in ordinary least squares regression, where the absence of regularization can lead to highly sensitive and unreliable coefficient estimates. Ridge regression (RR) \cite{h1970} introduces an $L_{2}$ penalty term controlled by the hyperparameter $\lambda$, which stabilizes estimates and enhances predictive performance. This example evaluates them efficacy under extreme multicollinearity conditions. 
\end{example}

This example uses synthetic data designed to simulate severe multicollinearity.
Assuming the number of samples is $m$ and the number of features is $n$, we first use the built-in MATLAB functions \texttt{orth} and \texttt{randn} to construct matrices $U_X\in\mathbb{R}^{m\times n}$ and $V_X\in\mathbb{R}^{n\times n}$ with orthonormal columns, and then construct an approximate low-rank singular value matrix: 
\begin{center}
\texttt{$S$ = diag([randn(1,$r_{\epsilon}$),randn(1,$n-r_{\epsilon}$)*1e-8])},
\end{center}
or a strictly low-rank singular value matrix:
\begin{center}
\texttt{$S$ = diag([randn(1,$r$),zeros(1,$n-r$)])}.
\end{center}
Then we construct the matrix $X=U_XSV_X^{\mathrm{H}}$.
The true coefficient vector $\beta$ contains uniformly distributed values in $[-1,1]$, while the response variable $y = X\beta + \varepsilon $ incorporates Gaussian noise ($\sigma=0.05$) to simulate measurement errors. The implementation includes intercept handling through matrix augmentation, ensuring controlled evaluation of regression methods' numerical stability. 
The estimated ridge regression coefficient $\beta$ is expressed as
$\hat{\beta} = (X^{\mathrm{H}}X+\lambda I)^{-1}X^{\mathrm{H}}y$,
where $\lambda\geq0$. We selected the HKB estimator to determine the parameter $\lambda$ \cite{a2025}, and its formula is given by
\[\lambda_{HKB} = \frac{n\tilde{\sigma}^{2}}{\sum_{i=1}^{n}\hat{\alpha}_{i}^{2}}.\]
Here, \\
$\hat{\alpha}_{i}$: Estimated regression coefficient for a predictor variable.\\
$\tilde{\sigma}^{2}$: Estimated error variance, measuring unexplained variability.

We use four methods to predict the coefficient vector $\beta$ for the estimator $\lambda$, namely the built-in function matrix inverse (BII) by MATLAB command \texttt{inv}, Algorithm \ref{Alg.GBRRI} (GRI), Algorithm 4.2 in \cite{9}, and randQB-EI in \cite{yu2018}.
We selected two key indicators, mean square error (MSE) and calculation time (seconds), to evaluate these methods.
The experimental results are shown in Table \ref{Table:RR}.

\begin{itemize}
  \item Under different data sizes and degrees of ill-conditioning, the RR by Algorithm \ref{Alg.GBRRI} achieves a 2-6 times reduction in computation time compared to RR by other inversion methods. This improvement is particularly beneficial for large-scale and high-dimensional data.
  \item When computing the matrix inverse, the RR by Algorithm \ref{Alg.GBRRI} keeps only the first $\epsilon$-rank column vectors from the range of $X$. It maintains almost the same prediction accuracy as RR by other inversion methods.
\end{itemize}

\begin{table}[H]
  \centering\scriptsize
  \renewcommand{\arraystretch}{1.2} 
  \caption{Multicollinearity data regression analysis indicator results}\label{Table:RR}
  \begin{tabular}{cclcc}
\toprule[1.2pt]											
$m$	&	$n$	&	Methods	&		Time(s)		&	MSE	\\
\midrule											
\multirow{4}*{5000}	&	\multirow{4}*{4000}	&	BII	&		0.7091		&	0.002387020	\\
	&		&	GRI	&	\textbf{	0.1409	}	&	0.002386973	\\
	&		&	ARRF	&		0.9693		&	0.002387020	\\
	&		&	randQB-EI	&		0.3581		&	0.002387020	\\
\midrule											
\multirow{4}*{10000}	&	\multirow{4}*{9500}	&	BII	&		6.9344		&	0.002523389	\\
	&		&	GRI	&	\textbf{	1.0160	}	&	0.002523398	\\
	&		&	ARRF	&		9.3122		&	0.002523389	\\
	&		&	randQB-EI	&		2.9679		&	0.002523389	\\
\midrule											
\multirow{4}*{7000}	&	\multirow{4}*{6000}	&	BII	&		2.0600		&	0.002742814	\\
	&		&	GRI	&	\textbf{	0.4927	}	&	0.002742811	\\
	&		&	ARRF	&		2.9096		&	0.002742814	\\
	&		&	randQB-EI	&		1.2679		&	0.002742814	\\
\bottomrule[1.2pt]											
  \end{tabular}\\
\end{table}

\section{Concluding remarks}
In this paper, for accelerating the SVD computation of a nearly low-rank matrix and the matrix inverse compitation, we propose definitions of $\epsilon$-rank and random re-normalization matrix. We then present random re-normalization theory for accelerating nearly low-rank matrix SVD and matrix inversion computations, which is an attempt of recombination of matrix columns, through an action of random transformation matrix to determine matrix $\epsilon$-rank automatically. Those methods use random sampling to identify a subspace that captures
most of the action of a matrix. The input matrix is then compressed either explicitly or
implicitly to this subspace, and the reduced matrix is manipulated deterministically to
obtain the desired low-rank factorization. Finally, we provide a variety of simulations and applications to access performance of the new theory and methods.
The simulations and applications uniformly show that the new theory and methods can be successfully applied in accelerating matrix SVD and matrix inversion computations.

\appendix
\section{Invariance of rank under multiplication by full-rank matrices}
\begin{lemma}\label{p2}
Let $B\in\mathbb{C}^{n\times d}$. If $A \in \mathbb{C}^{m\times n}$ has full column rank and $C\in\mathbb{C}^{d\times s}$ has full row rank, then we have $\mathrm{rank}(ABC)=\mathrm{rank}(B)$. 
\end{lemma}
\if1\blind{
\begin{proof}
Let SVDs of matrices $A$ and $C$ be
\begin{eqnarray*}
A=U_{A}\left[
 \begin{array}{c}
\Sigma_{A} \\
O_{(m-n)\times n}
 \end{array}
 \right]V_{A}^{\mathrm{H}}, \quad
C=U_{C}\left[
 \begin{array}{c}
\Sigma_{C},O_{d\times (s-d)}
 \end{array}
 \right]V_{C}^{\mathrm{H}},
\end{eqnarray*}
where $U_{A}\in\mathbb{U}_{m}, V_{A}\in\mathbb{U}_{n}, U_{C}\in\mathbb{U}_{d},V_{C}\in\mathbb{U}_{s}$, and both $\Sigma_{A}\in\mathbb{R}^{n\times n}$ and $\Sigma_{C}\in\mathbb{R}^{d\times d}$ are nonsingular. Then
\[
ABC=U_{A}\left[
 \begin{array}{c}
\Sigma_{A} \\
O_{(m-n)\times n}
 \end{array}
 \right]V_{A}^{\mathrm{H}} B U_{C}\left[
 \begin{array}{c}
\Sigma_{C},O_{d\times (s-d)}
 \end{array}
 \right]V_{C}^{\mathrm{H}},
\]
which implies that
\begin{eqnarray}
\mbox{rank}(ABC)&=&\mbox{rank}\left(\left[
 \begin{array}{c}
\Sigma_{A} \\
O_{(m-n)\times n}
 \end{array}
 \right]V_{A}^{\mathrm{H}} B U_{C}\left[
 \begin{array}{c}
\Sigma_{C},O_{d\times (s-d)}
 \end{array}
 \right]\right)\nonumber\\
 &=&\mbox{rank}\left(\left[
 \begin{array}{cc}
\Sigma_{A} V_{A}^{\mathrm{H}} B U_{C} \Sigma_{C}&O_{n\times (s-d)} \\
O_{(m-n)\times d}&O_{(m-n)\times (s-d)}
 \end{array}
 \right]\right)\nonumber\\
&=&\mbox{rank}(
\Sigma_{A} V_{A}^{\mathrm{H}}BU_{C} \Sigma_{C}).\nonumber
\end{eqnarray}
Since $\Sigma_{A},V_{A},U_{C}$ and $\Sigma_{C}$ are nonsingular, then
\begin{eqnarray*}
\mbox{rank}(ABC)=\mbox{rank}(
\Sigma_{A} V_{A}^{\mathrm{H}}BU_{C}\Sigma_{C})
=\mbox{rank}(B).
\end{eqnarray*}
\end{proof}
} \fi

\section{Expectation of Norms}
We recall the following result on the expected Frobenius norm of a scaled Gaussian rmatrix \cite[Proposition A.1]{9}.
\begin{lemma} \label{sgt1}
Fix real matrices $S,T$ and draw a real Gaussian random matrix
$\Omega$. Then $\mathbb{E}\| S\Omega T\|_{\mathrm{F}}^{2}=\| S\|_{\mathrm{F}}^2\| T\|_{\mathrm{F}}^2$.
\end{lemma}

The following corollary gives the complex version of Lemma \ref{sgt1}. The proof can be easily established by using the real representation of a complex matrix and unitary invariance of Gaussian matrices and therefore we omit it. 
\begin{corollary} \label{cor:sgt1} 
Fix complex matrices $S,T$ and draw a complex Gaussian random matrix
$\Omega$. Then $\mathbb{E}\| S\Omega T\|_{\mathrm{F}}^{2}=2\| S\|_{\mathrm{F}}^2\| T\|_{\mathrm{F}}^2$.
In particular, if $\Omega$ is real, then  $\mathbb{E}\| S\Omega T\|_{\mathrm{F}}^{2}=\| S\|_{\mathrm{F}}^2\| T\|_{\mathrm{F}}^2$.
\end{corollary}


\section{Proof of Lemma \ref{Thm.Expect}}
We first recall the following result on the eigenvalue inequalities of a Hermitian matrix \cite[Theorem 8.12]{zhang2011}.
\begin{lemma} \label{lem:eig}
Let $S,T\in\mathbb{C}^{n\times n}$ be Hermitian matrices. Then 
\[
\lambda_i(S)+\lambda_n(T)\le\lambda_i(S+T)\le \lambda_i(S)+\lambda_1(T),\quad i=1,\ldots,n.
\]
where $\lambda_i(S)$ denotes the $i$th largest eigenvalue of $S$.
\end{lemma}

{\em Proof of Lemma {\rm\ref{Thm.Expect}}}.
To ease the presentation, let $G=(I-Q_{k}Q_{k}^{\mathrm{H}})A$. The Frobenius norm of $A$ can be rewritten as
\begin{equation}\label{Eq.diffAhat}
   \|A\|_\mathrm{F}^2=\|(I-Q_{k}Q_{k}^\mathrm{H})A\|_\mathrm{F}^2+\|Q_{k}Q_{k}^\mathrm{H}A\|_\mathrm{F}^2=\|G\|_\mathrm{F}^2+\|\hat{A}\|_\mathrm{F}^2,
\end{equation}
where the first equality holds since the column spaces of $(I-Q_{k}Q_{k}^\mathrm{H})A$ and $Q_{k}Q_{k}^\mathrm{H}A$ are orthogonal. 
By Corollary \ref{cor:sgt1} we have
\begin{equation}\label{ieq.epG}
  \mathbb{E} b_{k+1,k+1}^2=\mathbb{E}(\|G\omega_{k+1}\|_2^2)=\|G\|_\mathrm{F}^2.
\end{equation}
\eqref{Eq.diffAhat}, together with \eqref{ieq.epG}, indicates that
\[
 \|A\|_{\mathrm{F}}^2-\|\hat{A}\|_{\mathrm{F}}^2=\|G\|_{\mathrm{F}}^2=\mathbb{E}b_{k+1,k+1}^2.
\]
Therefore,
\begin{equation}\label{ieq:aat}
\frac{\| A-\hat{A}\|_{\mathrm{F}}}{\| A\|_{\mathrm{F}}}=\frac{\sqrt{\| A\|_{\mathrm{F}}^2-\|\hat{A}\|_{\mathrm{F}}^2}}{\| A\|_{\mathrm{F}}}  \leq \sqrt{\epsilon}
\end{equation}
holds if and only if $\mathbb{E} b_{k+1,k+1}^2\leq\| A\|_{\mathrm{F}}^{2}\epsilon$. 
In this case, from \eqref{ieq:aat} we have
\begin{equation}\label{ieq:aat-sg}
\frac{\sum_{i=1}^{k}\sigma_{i}^{2}(\hat{A})}{\sum_{i=1}^{p}\sigma_{i}^{2}(A)}=\frac{\|\hat{A}\|_{\mathrm{F}}^2 }{\| A\|_{\mathrm{F}}^2}\geq 1-\epsilon.
\end{equation}

On the other hand, since $Q_k$ has orthonormal columns, we have
\begin{equation}\label{eq:ag}
A^\mathrm{H}A=A^\mathrm{H}Q_kQ_k^\mathrm{H}A+A^\mathrm{H}(I-Q_kQ_k^\mathrm{H})A=\hat{A}^\mathrm{H}\hat{A}+G^\mathrm{H}G.
\end{equation}
By Lemma \ref{lem:eig} we have
\begin{equation}\label{ieq:aat-sg2}
\sigma_i^2(A)=\lambda_i(A^\mathrm{H}A)\ge \lambda_i(\hat{A}^\mathrm{H}\hat{A})+\lambda_n(G^\mathrm{H}G)\ge \lambda_i(\hat{A}^\mathrm{H}\hat{A})=\sigma_i^2(\hat{A}),
\end{equation}
for  $i=1,\ldots,k$ since $G^\mathrm{H}G$ is Hermitian and positive semidefinite.
From \eqref{ieq:aat-sg} and \eqref{ieq:aat-sg2} we have
\[
\frac{\sum_{i=1}^{k}\sigma_{i}^{2}(A)}{\sum_{i=1}^{p}\sigma_{i}^{2}(A)}\geq 1-\epsilon.
\]
By Definition \ref{Ass1}, this implies that the $\epsilon$-rank of $A$ is $r_{\epsilon}=k$. 
%
{\hfill $\fbox{}$ \vspace*{5mm}} 

\section{Proof of Lemma \ref{vq0}}
As we know, a random variable $\xi$  is called sub-gaussian if $\mathbb{E}\exp(\xi^2/K^2)\le 2$ for some $K>0$. Also, define the sub-gaussian norm of $\xi$ as
\[ \|\xi\|_{\psi_2} = \sup_{p\geq1} p^{-1/2}(\mathbb{E}|X|^p)^{1/p}.\]

We now recall a complex version of the Hanson-Wright inequality for quadratic forms in sub-gaussian random variables \cite[Theorem 1.1]{rudelson2013hanson}.
\begin{lemma}\label{Lem.H-W}
 Let $\tau=[\tau_1,\tau_2,\ldots,\tau_n]$ be a real random vector, where  $\tau_i$ are independent sub-gaussian variables  satisfying $\mathbb{E} \tau_i=0$ and $\|\tau_i\|_{\psi_2} \leq K$. Let $C\in\mathbb{C}^{n\times n}$ be a Hermitian matrix. Then, for every $t >0$,
 \[
 P(|\tau^{\mathrm{T}}C\tau-\mathbb{E}\tau^{\mathrm{T}}C\tau|>t) \leq 2\exp\Big(-c\min\Big\{\frac{t^2}{K^4\|A\|_\mathrm{F}^2},\frac{t}{K^2\|A\|}\Big\}\Big),
 \]
where $c$ is an absolute constant.
\end{lemma}

{\em Proof of Lemma {\rm \ref{vq0}}}.
Let $G=(I-Q_kQ_k^{\mathrm{H}})A$. 
We first derive an upper bound of $\mathbb{E}(b_{k+1,k+1}^2)$. Using Equation \eqref{eq:ag} and Lemma \ref{lem:eig} we have
\begin{equation}\label{ieq:ag-1}
\sigma_i^2(A)=\lambda_i(A^\mathrm{H}A)\ge \lambda_i(G^\mathrm{H}G) + \lambda_n(\hat{A}^\mathrm{H}\hat{A})\ge  \lambda_i(G^\mathrm{H}G) =\sigma_i^2(G), 
\end{equation}
for $i=1,\ldots,p-k$ since $\hat{A}^\mathrm{H}\hat{A}$ is Hermitian and positive semidefinite.
The diagonal element $b_{k+1,k+1}^2$ can thus be rewritten as
\[
b_{k+1,k+1}^2=\|G\omega_{k+1}\|_2^2=\omega_{k+1}^{\mathrm{T}}G^{\mathrm{H}}G\omega_{k+1}.
\]
We note that  a Gaussian random variable $\xi$ is a sub-gaussian random variable with $\|\xi\|_{\psi_2} \le K$ for some $K>0$. By Lemma \ref{Lem.H-W}   we haveB
\begin{eqnarray}\label{Eq.Ebk}
&& P(|b_{k+1,k+1}^2-\mathbb{E}\;b_{k+1,k+1}^2| \leq t) = P(| \omega_{k+1}^{\mathrm{T}}G^{\mathrm{H}}G\omega_{k+1} 
-\mathbb{E}\omega_{k+1}^{\mathrm{T}}G^{\mathrm{H}}G\omega_{k+1} | \leq t)  \nonumber\\
&&\ge 1- 2\exp\Big(-c\min\big\{\frac{t^2} {K^4\|G^{\mathrm{H}}G\|_\mathrm{F}^2}, \frac{t}{K^2\|G^{\mathrm{H}}G\|}\big\}\Big).
\end{eqnarray}
By setting $t=\epsilon\|A\|_\mathrm{F}^2/2$, we have
\begin{equation}\label{Eq.Ebkprob}
\frac{t^2}{\|G^{\mathrm{H}}G\|_\mathrm{F}^2} = \frac{\epsilon^2\|A\|_\mathrm{F}^4}{4\|G^{\mathrm{H}}G\|_\mathrm{F}^2}
\ge \frac{\epsilon^2\|A\|_\mathrm{F}^4}{4\|G\|_\mathrm{F}^4}
 \geq \frac{ \epsilon^2 \|A\|_\mathrm{F}^2}{4\|G\|_\mathrm{F}^2} \geq \frac{\epsilon^2 \sum_{i=1}^{p}\sigma_i^2(A) }{4\sum_{i=1}^{p-k} \sigma_i^2(A)},
\end{equation}
where the first inequality uses the fact that $\|G^{\mathrm{H}}G\|_\mathrm{F}\le \|G\|_\mathrm{F}^2$, the second inequality uses the fact that $\|A\|_\mathrm{F}^2\ge\|G\|_\mathrm{F}^2$, and the last inequality uses \eqref{ieq:ag-1}.

From    \eqref{ieq.epG}, \eqref{Eq.Ebk}, and \eqref{Eq.Ebkprob},
\[
\|A-\hat{A}\|_{\mathrm{F}}^2=\|G\|_\mathrm{F}^2=\mathbb{E} b_{k+1,k+1}^2 \leq b_{k+1,k+1}^2+\frac{1}{2} \|A\|_\mathrm{F}^2\epsilon
\]
holds with probability at least $1-2 \exp(-c\epsilon^2 \sum_{i=1}^{p}\sigma_i^2(A) / (4K^4\sum_{i=1}^{p-k} \sigma_i^2(A)))$. Thus, if $b_{k+1,k+1}^2 \leq \|A\|_\mathrm{F}^2\epsilon/2$, then the condition \eqref{ieq:aat} holds with probability at least $1-2 \exp(-c\epsilon^2 \sum_{i=1}^{p}\sigma_i^2(A) / (4K^4\sum_{i=1}^{p-k} \sigma_i^2(A)))$. This, together with  \eqref{ieq:aat-sg2}, indicates that  
$\sum_{i=1}^{k}\sigma_{i}^{2}(\hat{A})/\sum_{i=1}^{n}\sigma_{i}^{2}(A)$ $\geq 1-\epsilon$ and $r_{\epsilon}=k$ with probability at least $1-2 \exp(-c\epsilon^2 \sum_{i=1}^{p}\sigma_i^2(A) / (4K^4\sum_{i=1}^{p-k} \sigma_i^2(A)))$. 
{\hfill $\fbox{}$ \vspace*{5mm}}

\section{Proof of Theorem \ref{l2}}
We first introduce a well-known Sherman-Morrison-Woodbury formula in Lemma \ref{l1}.
\begin{lemma}\label{l1}
Let $A\in\mathbb{C}^{n\times n}$ be nonsingular and $U,V\in\mathbb{C}^{n\times k}$ {\rm(}$k\le n${\rm)}. If $I_{k}+V^{\mathrm{H}}A^{-1}U$ is nonsingular, then $A+UV^{\mathrm{H}}$ is nonsingular and
\[
(A+UV^{\mathrm{H}})^{-1}=A^{-1}-A^{-1}U(I_{k}+V^{\mathrm{H}}A^{-1}U)^{-1}V^{\mathrm{H}}A^{-1}.
\]
\end{lemma}

\em Proof of Theorem {\rm\ref{l2}}. Let $[Q_{k},Q_{k}^{\perp}]\in\mathbb{U}_{m}$. By Algorithm \ref{alg:qr} we know that $A$ is approximated by $Q_{k}Q_{k}^{\mathrm{H}}A$. It follows that
\begin{eqnarray}
\lambda I_{m}+AA^{\mathrm{H}}&\approx&\lambda I_{m}+Q_{k}Q_{k}^{\mathrm{H}}AA^{\mathrm{H}}Q_{k}Q_{k}^{\mathrm{H}}\nonumber\\
&=&\lambda Q_{k}Q_{k}^{\mathrm{H}}+\lambda Q_{k}^{\perp}(Q_{k}^{^\bot })^\mathrm{H}+Q_{k}BB^{\mathrm{H}}Q_{k}^{\mathrm{H}}\nonumber\\
&=&[Q_{k},Q_{k}^{\perp}]\left[
\begin{array}{cc}
\lambda I_{k}+BB^{\mathrm{H}} &O_{k\times (m-k)} \\
O_{(m-k)\times k} &\lambda I_{m-k}
\end{array}
\right]\left[
\begin{array}{c}
Q_{k}^{\mathrm{H}}  \\
(Q_{k}^{\bot})^\mathrm{H}
\end{array}
\right].\nonumber
\end{eqnarray}
Then
\begin{eqnarray}
(\lambda I_{m}+AA^{\mathrm{H}})^{-1}&\approx&[Q_{k},Q_{k}^{\perp}]\left[
\begin{array}{cc}
(\lambda I_{k}+BB^{\mathrm{H}})^{-1} &O_{k\times (m-k)} \\
O_{(m-k)\times k} &\frac{1}{\lambda} I_{m-k}
\end{array}
\right]\left[
\begin{array}{c}
Q_{k}^{\mathrm{H}}  \\
(Q_{k}^{\bot})^\mathrm{H}
\end{array}
\right]\nonumber\\
&=&Q_{k}(\lambda I_{k}+BB^{\mathrm{H}})^{-1}Q_{k}^{\mathrm{H}}+\frac{1}{\lambda}Q_{k}^{\perp}(Q_{k}^{\bot})^\mathrm{H}\nonumber\\
&=&Q_{k}(\lambda I_{k}+BB^{\mathrm{H}})^{-1}Q_{k}^{\mathrm{H}}+\frac{1}{\lambda}I_{m}-\frac{1}{\lambda}Q_{k}Q_{k}^{\mathrm{H}}\nonumber\\
&=&\frac{1}{\lambda}I_{m}+Q_{k}\Big((\lambda I_{k}+BB^{\mathrm{H}})^{-1}-\frac{1}{\lambda}I_{k}\Big)Q_{k}^{\mathrm{H}}.\nonumber
\end{eqnarray}

The inverse of $\lambda I_{n}+A^{\mathrm{H}}A$ can also be immediately provided through a similar technique. Specifically, it follows that
\[
\lambda I_{n}+A^{\mathrm{H}}A\approx\lambda I_{n}+A^{\mathrm{H}}Q_{k}Q_{k}^{\mathrm{H}}A=\lambda I_{n}+B^{\mathrm{H}}B.
\]
Then by Lemma \ref{l1} we have
\begin{eqnarray*}
(\lambda I_{n}+B^{\mathrm{H}}B)^{-1}\approx\frac{1}{\lambda}I_{n}-\frac{1}{\lambda^{2}}B^{\mathrm{H}}(I_{k}+\frac{1}{\lambda}BB^{\mathrm{H}})^{-1}B.
\end{eqnarray*}
{\hfill $\fbox{}$}

\bibliographystyle{siam} 
\bibliography{reference} 

\end{document}